\documentclass[11pt,a4paper]{article}
\pdfoutput=1
\usepackage{etex}
\usepackage[utf8]{inputenc} 
\usepackage{verbatim}
\usepackage{graphicx}
\usepackage{color}
\usepackage{epsf}
\usepackage{xspace}
\usepackage{subfigure}
\usepackage{url}
\usepackage{a4wide}
\usepackage{amsmath}
\usepackage{amssymb}
\usepackage{amsthm}

\usepackage{authblk}

\usepackage{algorithm}
\usepackage{algpseudocode} 
\makeatletter
\algrenewcommand\ALG@beginalgorithmic{\small}
\makeatother

\usepackage{longtable}

\usepackage[usenames,dvipsnames]{pstricks}
\usepackage{epsfig}
\usepackage{pst-grad} % For gradients
\usepackage{pst-plot} % For axes
\usepackage{pst-node}%for graphs 

\usepackage{tikz}
\usetikzlibrary{positioning, decorations.pathreplacing, shapes.symbols}
\tikzstyle{vertex}=[circle,minimum size=15pt,inner sep=0pt, draw=black]
\usepackage{pgfplots}
\pgfplotsset{compat=newest}
\pgfplotsset{
tick label style={font=\small},
label style={font=\small},
legend style={font=\footnotesize}
}

\newtheorem{de}{Definition} 
\newtheorem{nota}[de]{Notation}
\newtheorem{lemma}[de]{Lemma} 
\newtheorem{theo}[de]{Theorem} 
\newtheorem{cor}[de]{Corollary} 
\newtheorem{remark}[de]{Remark} 
\newtheorem{example}[de]{Example}

\newcommand{\R}{\mathbb{R}}%{{\rm I\hspace{-0.15em}R}}
\newcommand{\Z}{\mathbb{Z}}%{{\rm Z\hspace{-0.35em}Z}}
\newcommand{\N}{\mathbb{N}}%{{\rm I\hspace{-0.15em}N}}
\newcommand{\cU}{\mathcal{U}}

\newcommand{\cSP}{\mathcal{P}}
\newcommand{\cMSP}{\mathcal{MP}}
\newcommand{\cGMSP}{\mathcal{GMP}}

\newcommand{\X}{\mathcal{X}}

\newcommand{\cX}{\mathcal{X}}

\newcommand{\cost}{c} % Kantenkostenvektor ohne Indizes 
\newcommand{\ce}[1]{\cost_{#1}} %Kantenkosten #1=Kante
\newcommand{\cex}[2]{\cost_{#1}^{#2}} %Kantenkosten mit Szenario: #1=Kante, #2=Szenario, Alternativen zB: z_(#1,#2),c_{#1}^{#2}, c^{#2}(#1)
\newcommand{\cexi}[3]{\cost_{#1,#3}^{#2}} %Kantenkosten mit Szenario und Kriterium: #1=Kante, #2=Szenario #3=Kriterium
\newcommand{\cei}[2]{\cost_{#1,#2}} %Kantenkosten mit Index ohne Szenario: #1=Kante, #2=Kriterium
\newcommand{\nomei}[2]{\hat{\cost}_{#1,#2}} %nominale Kantenkosten mit Index: #1=Kante, #2=Kriterium
\newcommand{\nome}[1]{\hat{\cost}_{#1}} %nominale Kantenkosten #1=Kante
\newcommand{\nom}{\hat{\cost}} %nominale Kantenkosten 
\newcommand{\kmax}[1]{#1\text{-}\max \ } % #1-größtes Element 
\newcommand{\kmaxs}[2]{#1\text{-}\max_#2\ } % #1-größtes Element der Menge #2

\newcommand{\zBS}{g} %Zfkt d. Subprobleme des Bertsimas-Sim-Algorithmus
\newcommand{\ZBS}{OPT} %Menge d Lsg d Subprobleme des Bertsimas-Sim-Alg
\newcommand{\zRC}{z^{\text{R}}} %Zfkt des robust counterparts
\newcommand{\zLA}{z^{\text{L}}} %Zfkt des Labeling-Problems im Labeling&Filtering-Algorithmus
\newcommand{\DSA}{DSA}
\newcommand{\La}{Bottleneck approach}
\newcommand{\la}{bottleneck approach}
\newcommand{\Lab}{LSA}
\newcommand{\idelta}{\bar{\delta}}
\usepackage[normalem]{ulem} 
\usepackage{color}

\begin{document}  

\title{Multi-objective minmax robust combinatorial optimization with cardinality-constrained uncertainty}

\author[a]{Andrea Raith}
\author[b]{Marie Schmidt}
\author[c]{Anita Sch\"obel}
\author[c,*]{Lisa Thom}

\affil[a]{Department of Engineering Science, The University of Auckland, postal address: Private Bag 92019, Auckland 1142, New Zealand, email address: a.raith@auckland.ac.nz}
\affil[b]{Department of Technology and Operations Management, Rotterdam School of Management, Erasmus University Rotterdam, postal address: PO Box 1738, 3000 DR Rotterdam, The Netherlands, email address: schmidt2@rsm.nl}
\affil[c]{Institut f\"ur Numerische und Angewandte Mathematik, Georg-August-Universit\"at G\"ottingen, postal address: Lotzestr. 16-18, 37083 G\"ottingen, Germany, email addresses: schoebel@math.uni-goettingen.de (Anita Sch\"obel), l.thom@math.uni-goettingen.de (Lisa Thom)}
\affil[*]{Corresponding author, email address: l.thom@math.uni-goettingen.de}

\renewcommand\Affilfont{\raggedright \small}

\date{}

 \maketitle

  \abstract{  
 In this paper we develop two approaches to find minmax robust efficient solutions for multi-objective combinatorial optimization problems with cardinality-constrained uncertainty. First, we extend an algorithm of Bertsimas and Sim (2003) for the single-objective problem to multi-objective optimization. We propose also an enhancement to accelerate the algorithm, even for the single-objective case, and we develop a faster version for special multi-objective instances. Second, we introduce a deterministic multi-objective problem with sum and bottleneck functions, which provides a superset of the robust efficient solutions. Based on this, we develop a label setting algorithm to solve the multi-objective uncertain shortest path problem. We compare both approaches on instances of the multi-objective uncertain shortest path problem originating from hazardous material transportation. 
}

\section{Multi-objective robust combinatorial optimization }\label{section_intro}

\subsection{Introduction}
Two of the main difficulties in applying optimization techniques to real-world problems are that several (conflicting) objectives may exist and that parameters may not be known exactly in advance. In multi-objective optimization several objectives are optimized simultaneously by choosing solutions that cannot be improved in one objective without worsening it in another objective. 
Robust optimization hedges against (all) possible parameter values, e.g., by assuming the worst case for each solution (minmax robustness). \\ Often it is assumed that the uncertain parameters take any value from a given interval or that discrete scenarios are given. A survey on robust combinatorial optimization with these uncertainty sets is given in \cite{aissi2009min}. Based on the interval case, Bertsimas and Sim propose in \cite{bertsimas2004price} to consider scenarios where only a bounded number of parameters differ from their expected value (cardinality-constrained uncertainty). This leads to less conservative solutions that are of high practical use. In \cite{bertsimas2003robust} an algorithm is provided to find robust solutions for combinatorial optimization problems under this kind of uncertainty. \\
Only recently have robust optimization concepts for multi-objective problems been developed. A first extension of minmax robustness for several objectives was introduced in  \cite{kuroiwa2012robust} and \cite{fliege2014robust}. They consider the uncertainties in the objectives independently of each other. Ehrgott et al. developed another extension of minmax robustness \cite{EIS14}, in which they include the dependencies between the objectives, and which was generalized in \cite{IKKST13}. These concepts have been extensively applied, e.g., in portfolio management \cite{fliege2014robust}, in game theory \cite{Yu2013} and in the wood industry \cite{ide2015application}. An overview on multi-objective robustness, including further robustness concepts, is given in \cite{IdeSchoe13} and \cite{wiecek2016robust}. Newest developments in this field include \cite{Chuong2016} and \cite{KDD2016}. Cardinality constrained uncertainty has been extended to multi-objective optimization in \cite{DKW2012} (only for uncertain constraints) and \cite{Hassanzadeh13} (for uncertain objective functions and constraints).\\
To the best of our knowledge, only Kuhn et al. have developed a solution algorithm for multi-objective uncertain combinatorial optimization problems \cite{robipa}. They consider problems with two objectives, of which only one is uncertain, with discrete and polyhedral uncertainty sets.
\\
 In this paper, however, we consider problems with arbitrarily many objectives of which all may be uncertain. The main contributions of this paper are that we develop two solution approaches for multi-objective combinatorial optimization problems with cardinality-constrained uncertainty and derive specific algorithms for the multi-objective uncertain shortest path problem.\\
 The remainder of this paper is structured as follows: In Section~\ref{section_intro} we give a short introduction to multi-objective robust optimization. We present two solution approaches for multi-objective combinatorial optimization problems with cardinality-constrained uncertainty in Section~\ref{section_algos}: In Section~\ref{section_DSA} we extend an algorithm from \cite{bertsimas2003robust} to multi-objective optimization and, additionally, propose an acceleration for both the single-objective and the multi-objective case and a faster version for multi-objective problems with a special property. In Section~\ref{section_la} we introduce a second approach and show how it can be applied to solve the multi-objective uncertain shortest path problem as an example. In Section~\ref{section_experiments}, we compare our methods on instances of the multi-objective uncertain shortest path problem originating from hazardous material transportation.

 \subsection{Multi-objective optimization}
First, we will give a short introduction to multi-objective optimization.
\begin{de}
 Given a set $\mathcal{X}$ of feasible solutions and $k$ objective functions $z_1,...,z_k: \mathcal{X}\rightarrow \R$ with $k \geq 2$, we call \[ \min_{x\in \X} z(x) = \begin{pmatrix} 
                                  z_1(x) \\ \vdots \\ z_k(x)
                                 \end{pmatrix}
                                 \]
a \emph{multi-objective optimization problem (MOP)}.
\end{de}
A solution that minimizes all objectives simultaneously does usually not exist. Therefore, we use the concept of \emph{efficient solutions}.
\begin{nota}
 For two vectors $y^1,y^2 \in \R^k$ we use the notation
 \begin{align*}
  y^1 \leq y^2 & \Leftrightarrow y^1_i \leqq y^2_i \text{ for }i = 1,...,k \text{ and } y^1 \neq y^2, \\
  y^1 \leqq y^2 & \Leftrightarrow y^1_i \leqq y^2_i \text{ for }i = 1,...,k.
 \end{align*}
\end{nota}
In the following, we will only use the symbols $<$ (strictly less than) and $\leqq$ (less than or equal to) to compare scalars.
\begin{de}
 A solution $x' \in \cX$ \emph{dominates} another solution $x \in \X$ if $z(x') \leq z(x)$. We also say that $z(x')$ dominates $z(x)$. A solution $x\in\X$ is an \emph{efficient} solution, if there is no $x' \in \X$ such that $x'$ dominates $x$. Then $z(x)$ is called \emph{non-dominated}.
\end{de}
Solving a multi-objective optimization problem $\min\{z(x)=(z_1(x),...,z_k(x)): x\in \X\}$ means to find its efficient solutions.
\begin{de} \label{def_equivalent_complete}
 Two efficient solutions $x,x' \in \cX$ are called \emph{equivalent} if $z(x) = z(x')$. A set of efficient solutions $\bar{\cX}\subseteq \cX$ is called \emph{complete} if all $x\in \cX \setminus \bar{\cX}$ are either dominated by or equivalent to at least one $x' \in \bar{\cX}$. 
\end{de}

\subsection{Robust optimization}
We briefly introduce robust optimization for single-objective problems.\\
In robust optimization the uncertain input data is given as an \emph{uncertainty set} $\cU$, containing all possible \emph{scenarios} that can occur. For each scenario $\xi \in \cU$ we obtain a different instance of the optimization problem $\min_{x \in \X} z(x, \xi)$. 
\begin{de}
 Given a feasible set of solutions $\X$, an uncertainty set $\cU$ and an objective function $z: \X \times \cU \rightarrow \R$, we define an \emph{uncertain optimization problem} as the family of parameterized problems on $\X$ \[\left( \min_{x\in\X} z(x,\xi), \xi \in \cU \right).\]
\end{de}
We only consider problems with uncertainty in the objective function, not in the constraints. This is because, in the considered robustness concepts, a solution is only robust feasible if it is feasible for every scenario. This is reasonable for many combinatorial optimization problems, e.g., when choosing a path in a road or transportation network: If we decide on a path to take without knowing which scenario will occur, this path should at least exist for every scenario. Hence, we have deterministic constraints.\\
There are different \emph{robustness concepts} offering a definition of a \emph{robust solution} for an uncertain optimization problem, usually by defining a deterministic problem, called the \emph{robust counterpart} (see \cite{GoeSchoe13-AE} for an overview). The concept of \emph{minmax robustness}, also called \emph{strict} or \emph{worst case} robustness, seeks solutions, for which the worst possible objective value is minimized. The solutions can be found by solving the robust counterpart \[\min_{x \in \X} \sup_{\xi \in \cU} z(x,\xi).\]
The considered uncertainty set often strongly influences the solvability and the solution approaches.
A \emph{finite uncertainty set} consists of finitely many scenarios, whereas, in an \emph{interval uncertainty set}, the coefficients vary in intervals independently of each other. If the coefficients vary in intervals, but only a given number of coefficients may differ from their minimal values, we speak of \emph{cardinality-constrained uncertainty} \cite{bertsimas2003robust}.

\subsection{Multi-objective robust optimization}

If several objective functions and uncertainties in (some of) these functions are given, we obtain a \emph{multi-objective uncertain optimization problem}.
\begin{de}
 Given a feasible set of solutions $\X$, an uncertainty set $\cU$ and a multi-objective function $z: \X \times \cU \rightarrow \R^k$, the family of multi-objective optimization problems 
 \begin{align} \left( \min_{x\in\X} z(x,\xi), \xi \in \cU \right) \label{MO_uncertain}\end{align} 
 is called a \emph{multi-objective uncertain optimization problem}.
\end{de}
There are several definitions of robust efficiency for multi-objective uncertain problems (see, e.g., \cite{IdeSchoe13}). The concept of minmax robust optimality for single-objective uncertain problems has been generalized to several objectives in various ways, since the notion of \emph{worst case} is not clear in the multi-objective case. An intuitive concept, introduced by Kuroiwa and Lee \cite{kuroiwa2012robust}, is to determine the worst case independently for each objective (see Definition~\ref{def_minmax_pointbased}). This yields a single vector for each solution and these vectors can be compared using the methods of multi-objective optimization.
\begin{de} \label{def_minmax_pointbased}
 A solution $x \in \X$ is \emph{robust efficient} for Problem (\ref{MO_uncertain}) if $x$ is an efficient solution for the robust counterpart
 \[ \min_{x\in \X} \zRC(x) = \begin{pmatrix}
                    \sup_{\xi \in \cU} z_1(x,\xi)\\
                    \vdots\\
                    \sup_{\xi \in \cU} z_k(x,\xi)
                   \end{pmatrix}.\]
 \end{de}

\begin{remark} In this paper, we only consider uncertainty sets where the uncertainties in the objectives are independent of each other. That means that robust efficiency, as defined in Definition~\ref{def_minmax_pointbased}, is the same as point-based and set-based minmax robust efficiency defined in \cite{EIS14}. Therefore, all results shown in this paper are valid for both concepts.\end{remark}
Analogously to Definition~\ref{def_equivalent_complete} we define:
\begin{de}
 Two robust efficient solutions $x,x' \in \cX$ are called \emph{equivalent} if $\zRC(x) = \zRC(x')$. A set of robust efficient solutions $\bar{\cX} \subseteq \cX$ is called \emph{complete} if all $x\in \cX \setminus \bar{\cX}$ are either dominated w.r.t.~$\zRC$ or equivalent to at least one $x' \in \bar{\cX}$. %If a set of efficient solutions is complete and doesn't contain any equivalent solutions it is called \emph{minimal}.
\end{de}

\subsection{Multi-objective robust combinatorial optimization}
An instance $(E,Q,\cU,\cost)$ of a multi-objective uncertain combinatorial optimization problem is given by a finite element set $E$, a set $Q \subseteq 2^{|E|}$ of feasible solutions, which are subsets of $E$, an uncertainty set $\cU$ and a function $\cost$, that assigns a \emph{cost} vector $\cex{e}{\xi} = (\cexi{e}{\xi}{1},...,\cexi{e}{\xi}{k}) $ to each element $e\in E$ and scenario $\xi \in \cU$. For each scenario $\xi$ the \emph{cost} $z(q,\xi)$ of a set $q$ with respect to $\xi$ is the sum of the costs of its elements. % $z(q,\xi)=\sum_{e\in q} \cex{e}{\xi}.\]
We aim to find a complete set of robust efficient solutions (according to Definition~\ref{def_minmax_pointbased}) for 
\begin{align*} \left( \min_{q \in Q} z(q,\xi)=\sum_{e \in q} \cex{e}{\xi}, \xi \in \cU \right),\end{align*} i.e., to find a complete set of efficient solutions for the robust counterpart 
\begin{align*}  \min_{q \in Q} \begin{pmatrix}
				  \max_{\xi \in \cU} \sum_{e \in q} \cexi{e}{\xi}{1}\\
				  \vdots\\
                                 \max_{\xi \in \cU} \sum_{e \in q} \cexi{e}{\xi}{k}
                               \end{pmatrix}.
\end{align*}

\subsection{Example: The multi-objective uncertain shortest path problem}\label{subsection_shortest_path}

 Consider a graph $G=(V,E)$ with node set $V$ and edge set $E$, a start node $s\in V$ and a termination node $t\in V$. Let $\cU$ be an uncertainty set and $\cost$ be a function that assigns a \emph{cost} or \emph{length} $\cex{e}{\xi} = (\cexi{e}{\xi}{1},...,\cexi{e}{\xi}{k}) $ to each edge $e\in E$ and scenario $\xi \in \cU$. For a path $q$ in $G$ and a scenario $\xi \in \cU$ the \emph{cost} or \emph{length} $z(q,\xi)$ of $q$ w.r.t. $\xi$ is obtained by following the path and adding up the costs $\cex{e}{\xi}$ of the edges traversed. \\ 
 We distinguish between \emph{simple paths}, which contain each node at most once and \emph{paths}, which may contain nodes and edges more than once. In the deterministic case, there always either exists a simple path being a shortest path, or no finite shortest path exists. On the contrary, robust shortest paths that contain a cycle but are not optimal without the cycle can exist, even in case of only one objective (see Example~\ref{example_robust_path_with_cycle}). \\ In the following we assume \emph{conservative} edge costs, i.e., every cycle $C$ has non-negative cost $z(C,\xi) \geq 0$ for each scenario $\xi \in \cU$ and objective $i=1,...,k$. Then, there always exists a complete set of robust efficient paths containing only simple paths and the \emph{multi-objective uncertain shortest path problem} is
  \begin{align*} \left( \min_{q \in Q} z(q,\xi)=\sum_{e\in q} \cex{e}{\xi}, \xi \in \cU \right)\end{align*}
  with $Q$ being the set of simple paths from $s$ to $t$ in $G$. Because simple paths do not contain any edge more than once, this is a combinatorial optimization problem. \\
 
The following single-objective example shows that, when edge costs are not conservative, we can indeed have robust shortest paths which contain cycles while no simple robust shortest path exists.
  \begin{example}\label{example_robust_path_with_cycle}
Let $G$ be a graph that consists of a simple path $q$ from $s$ to $t$ and a cycle $C$ connected to $q$ (Figure~\ref{figure_robust_path_with_cycle}). Let two scenarios $\xi_1, \xi_2$ be given and let the cost of $C$ be $z(C,\xi_1)= -1$ and $z(C,\xi_2)=0$ and the cost of $q$ be $z(q,\xi_1)= 3$ and $z(q,\xi_2)=2$. Let $q^i$ for $i \in \N$ denote the path that consists of $q$ and $i$ times the cycle $C$. Then,
\begin{align*}
 \max_{\xi \in \{\xi_1, \xi_2\}} z(q,\xi) = 3 > 2 = \max_{\xi \in \{\xi_1, \xi_2\}} z(q^1,\xi) = \max_{\xi \in \{\xi_1, \xi_2\}} z(q^i,\xi)  \ \forall \ i \geqq 1,
\end{align*}
and $q$ is not robust optimal, but $q^1$ is robust optimal.
 \end{example}
  \begin{figure}
      \begin{tikzpicture}
      \tikzstyle{vertex}=[circle,fill=black,minimum size=8pt,inner sep=0pt]
      \node (null) at (0,0) {};
      \node[vertex, label=left:{s}] (node1) at (5,0) {};
      \node[vertex] (node2)  at (7.5,0) {}; %\node[vertex, right of=node1] (node2)  {};
      \node[vertex, label=right:{t}] (node3)  at (10,0) {};
      
      \draw[->,line width=0.2mm] (node1) -- (node2) node [midway, above] (TextNode) {2 $\mid$ 1};
      \draw[->,line width=0.2mm] (node2) -- (node3) node [midway, above] (TextNode) {1 $\mid$ 1};%\draw[->,line width=0.2mm] (node2) to [bend left] (node2) ;
      \draw [->,line width=0.2mm]  (node2)  to [out=45,in=135,looseness=30] node [midway,above] {-1 $\mid$ 0} node [pos=0.75,left] {C}  (node2)  ;

      \draw[-,decorate, decoration={brace, mirror}, line width=0.18mm ] (5,-0.3) -- (10,-0.3) node [pos=0.5,below] {$q$};
   \end{tikzpicture}
   \caption{In Example~\ref{example_robust_path_with_cycle} every robust shortest path contains a cycle.}\label{figure_robust_path_with_cycle}
  \end{figure}
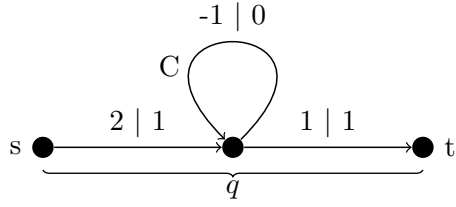
We use the following notation to specify subpaths.
\begin{nota}
 Let $q$ be a simple path and $v,w$ two nodes on $q$ ($v$ before $w$). Let then $q_{v,w}$ denote the part of $q$ from node $v$ to node $w$.
\end{nota}

\section{Algorithms for multi-objective combinatorial optimization problems with cardinality-constrained uncertainty}\label{section_algos} % oder $\Gamma$-uncertainty sets
The idea of cardinality-constrained uncertainty is to assume that the worst case will not happen for all edges simultaneously, e.g., there will not be an accident on every road of a transportation network at the same time. Therefore, only those scenarios are considered where no more than a given number of elements are more expensive than their minimum costs.
 Bertsimas and Sim were the first to introduce cardinality-constrained uncertainty for single-objective uncertain optimization problems \cite{bertsimas2003robust}. One possibility to extend this concept to multi-objective optimization is the following (see \cite{Hassanzadeh13}):
\begin{de} \label{def_mo_gamma_uncertainty}
For each element $e \in E$ and each objective $z_i$ let two real values $\nomei{e}{i}$ and $\delta_{e,i} \geqq 0$ be given. We assume that the uncertain cost $\cei{e}{i}$ can take any value in the interval $[\nomei{e}{i}, \nomei{e}{i} + \delta_{e,i}]$, with $\nomei{e}{i}$ being the undisturbed value, called the \emph{nominal} value. For each objective $z_i$ let an integer $\Gamma_i \leqq |E|$ be given. The \emph{cardinality-constrained uncertainty set} contains all scenarios, in which for each scenario $i$ at most $\Gamma_i$ elements differ from their nominal costs:
 \[ \cU:= \{ \cost \in \R^{|E|\times k}: \cei{e}{i} \in [ \nomei{e}{i}, \nomei{e}{i} + \delta_{e,i} ] \ \forall \  e\in E,\ \forall \ i =1,...,k, |\{ e: \cei{e}{i} > \nomei{e}{i}\}|\leqq \Gamma_i \ \forall \ i=1,...,k\} \] 
\end{de}

An instance of a multi-objective combinatorial optimization problem with cardinality-constrained uncertainty is hence given by $(E, Q,\nom, \delta, \Gamma=(\Gamma_1,...,\Gamma_k))$.\\

\subsection{Deterministic Subproblems Algorithm (\DSA{})}\label{section_DSA}
The algorithms in this subsection are built upon an algorithm by Bertsimas and Sim for single-objective cardinality-constrained uncertain combinatorial optimization problems \cite{bertsimas2003robust}, which we call Deterministic Subproblems Algorithm (\DSA{}).  Its idea is to find solutions for the uncertain problem by solving up to $|E|+1$ deterministic problems of the same type and comparing their solutions.\\ We describe first the algorithm of Bertsimas and Sim for single-objective problems and present several ways to reduce the number of subproblems to be solved (Section~\ref{subsection_detsubproblems_one_objective}). In Section~\ref{subsection_detsubproblems_objective_independen}, we show that \DSA{} can be adjusted for multi-objective problems with a special property. Lastly, we extend the algorithm for the general multi-objective case in Section~\ref{subsection_detsubproblems_multi-objective}.

\subsubsection{\DSA{} for one objective} \label{subsection_detsubproblems_one_objective}

 We first consider the single-objective uncertain problem $\left( \min_{q \in Q} z(q,\xi), \xi \in \cU \right)$ with \[\cU = \{ \cost \in \R^{|E|}: \ce{e} \in [ \nome{e}, \nome{e} + \delta_{e} ] \ \forall \  e\in E, |\{ e: \ce{e} > \nome{e}\}|\leqq \Gamma \}.\] 
 The worst case for a set $q \in Q$ with respect to this uncertainty is a scenario, 
where the costs of its $\Gamma$ elements with the largest intervals $\delta_e$ take their maximal value 
(resp.~all elements in $q$, if $q$ has less than $\Gamma$ elements).
Assume that the elements are ordered with respect to the interval length $\delta$, i.e., 
 \[\idelta_1:= \delta_{e_1} \geqq \idelta_2:= \delta_{e_2}\geqq... \geqq\idelta_{|E|}:= \delta_{e_{|E|}} \geqq \idelta_{|E|+1} := 0.\] 
 For each $l\in \{1,...,|E|+1\}$ we define the function $\zBS^l$ as follows \cite{bertsimas2003robust}:
 \[\zBS^l(q) :=   \sum_{e \in q} \nome{e} + \Gamma \cdot \idelta_{l} + \sum_{\substack{e_j \in q \\ j \leqq l}} (\delta_{e_j} - \idelta_{l}). \] 
The function $\zBS^l(q)$ is an approximation of the worst case costs of the set $q$. It contains
\begin{itemize}
 \item the nominal cost $\nome{e}$ for each element $e\in q$, which has to be paid also in the worst case, 
 \item $\idelta_l \cdot \Gamma$ since, in the worst case, the interval length $\delta_e$ has to be added to the costs for (at most) $\Gamma$ elements, 
 \item the positive summand $\max \{0, \delta_e - \idelta_l\}$ for each element $e \in q$ to account for all elements in the set with higher interval lengths than $\idelta_l$.
\end{itemize}
The idea of the algorithm of \cite{bertsimas2003robust} is to solve all problems 
\[ (\cSP(l))\ \min_{q\in Q} \zBS^l(q) \]
for $l=0,1,\ldots |E|+1$ and choose the best of the obtained solutions. 
This idea works due to the following two properties:
\begin{enumerate}
\item For every set $q$ and every $l\in \{0,\ldots,|E|+1\}$ we have that 
$\zBS^l(q)$ is always greater than or equal to the worst case cost $\zRC(q)$. 
\item For every set $q$ there exists some $l \in \{0,\ldots,|E|+1\}$ such that 
$\zBS^l(q)$ equals the worst case cost $\zRC(q)$.
\end{enumerate}
To show the first property, let $q$ be a set and let $\{ e_{a_1},\ldots, e_{a_h} \}$ be a subset of $h$ elements in $q$ with the largest cost intervals, where $h =\min\{|q|,\Gamma\}$. Then $\zRC(q)=\sum_{e \in q} \nome{e} + \sum_{j=1}^h \delta_{e_{a_j}}$ and we get
\[\zBS^l(q) \geqq \sum_{e\in q} \nome{e} + \sum_{j=1}^h \idelta_l + \sum_{j=1}^h \max \{0,\delta_{e_{a_j}} - \idelta_l\} \geqq  \zRC(q).\]
For the second property we show that for each set $q$ there exists at least one index $l$ with 
$\zBS^l(q) = \zRC(q)$: If $q$ has less than $\Gamma$ elements, then 
\[g^{|E|+1}(q) = \sum_{e\in q}\nome{e} + \Gamma \cdot 0 + \sum_{e\in q} (\delta_e - 0) = \zRC(q).\] 
If $q$ has at least $\Gamma$ elements, let $e_l$ be the element in $q$ with the $\Gamma$-th smallest index. Then the $\Gamma$ elements $\{ e_j \in q : j \leqq l\}$ have the largest cost intervals in $q$ and it follows that
\[ \zBS^{l}(q) = \sum_{e \in q} \nome{e} + \Gamma \cdot \idelta_{l} + \sum_{\substack{e_j \in q \\ j \leqq l}} (\delta_{e_j} - \idelta_{l}) = \sum_{e \in q} \nome{e} + \sum_{\substack{e_j \in q \\ j \leqq l}} \idelta_{l} + \sum_{\substack{e_j \in q \\ j \leqq l}} (\delta_{e_j} - \idelta_{l}) = \zRC(q).\]
Having these two properties, we see that a robust optimal solution $q^*$ is also optimal for problem $(\cSP(l))$ with $l : \zBS^{l}(q) = \zRC(q)$, since none of the other sets can have a better objective value. Therefore, at least one robust optimal solution will be found by the iterative algorithm. \\ The efficiency of the algorithm depends on the time complexity to solve the subproblems $(\cSP(l))$.
Because the summand $\Gamma \cdot \idelta_l$ is solution-independent, a solution for $(\cSP(l))$ 
can be found efficiently by solving a problem of the same kind as the underlying deterministic problem
with element costs
\begin{align*}  \ce{e_j}^l := \begin{cases}
         \nome{e_j} + (\delta_{e_j} - \idelta_{l}) &\text{ for } j < l  \\
         \nome{e_j} &\text{ for } j \geqq l.
        \end{cases} \tag{\textasteriskcentered}\label{bertsim_sp_lenghts}
\end{align*}
\\Algorithm~\ref{alg_Bert_Sim_structure} shows the basic structure of the algorithm by Bertsimas and Sim (\cite{bertsimas2003robust}). First, the elements are ordered with respect to their interval lengths, then the subproblems defined above are solved and finally the worst case values of all obtained solutions are compared to find the robust optimal ones. Because the solutions of each subproblem can be obtained by solving a deterministic problem of the same kind, this algorithm finds a robust optimal solution in polynomial time for many combinatorial optimization problems, e.g., for the minimum spanning tree and the shortest path problem.

 \begin{algorithm} 
 	\caption{Basic structure of \DSA{} (based on \cite{bertsimas2003robust})} \label{alg_Bert_Sim_structure}
	\begin{algorithmic}[1]
	 \Require an instance $I=(E, Q,\nom, \delta, \Gamma)$ of a single-objective cardinality-constrained uncertain combinatorial optimization problem
	\Ensure a robust efficient solution for $I$
	 \State\label{alg_step_sort} Sort $E$ w.r.t.~$\delta_e$ such that $\idelta_{1} := \delta_{e_1} \geqq \idelta_{2} := \delta_{e_2} \geqq ... \geqq \idelta_{{|E|}} \geqq \idelta_{{|E|+1}} := 0$.
	 \State\label{alg_step_determine} Determine $L:= \{1,...,|E|+1\}$.
	 \State\label{alg_step_solve} For all $l \in L$ find an optimal solution $q^l$ for $(\cSP(l))$.
	 \State\label{alg_step_compare} Compare the objective values $\zRC(q^l)$ for all $ l \in L$ to obtain a robust optimal solution. 
	\end{algorithmic}
\end{algorithm}

In the following, we show how Algorithm~\ref{alg_Bert_Sim_structure} can be enhanced. It is not necessary to solve all of the $|E|+1$ subproblems introduced above. The following three results from \cite{bertsimas2003robust, park2007note, lee2014short} can be used to reduce the number of subproblems (Lemma~\ref{lemma_alg_Bertsim_faster_1}): First, if two elements have the same interval length $\delta_e$, then their associated subproblems are equal. Second, the worst case cost of a set equals its objective value not only for the special subproblem shown above, but also for the next subproblem. Therefore, we do not miss any solutions if we only solve every second problem. Third, none of the first $\Gamma-1$ elements can be the one with the $\Gamma$-th smallest index for any set in $Q$, so their associated subproblems need not to be solved.  
\begin{lemma}[\cite{bertsimas2003robust, park2007note, lee2014short}] \label{lemma_alg_Bertsim_faster_1} 
The number of subproblems to be solved by {Algorithm~\ref{alg_Bert_Sim_structure}} can be reduced to at most ${\left\lceil \frac{|E|-\Gamma}{2} \right\rceil +1}$ in the following ways:
 \begin{enumerate}
 \item \label{lemma_alg_Bertsim_faster_1_a}  If there are several elements $e_l,...,e_{(l+h)}$ with the same interval length $\delta_{e_l} = ... = \delta_{e_{l+h}}$, only one of the subproblems $\cSP(l), ..., \cSP(l+h)$ needs to be solved \cite{bertsimas2003robust}. 
 \item \label{lemma_alg_Bertsim_faster_1_c} Only every second subproblem and the last subproblem need to be solved \cite{lee2014short}.
 \item \label{lemma_alg_Bertsim_faster_1_b} It is sufficient to start with the $\Gamma$-th subproblem \cite{park2007note}.
\end{enumerate}
\end{lemma} 
Using these results we can replace Step~\ref{alg_step_determine} of the basic structure with Algorithm~\ref{alg_Bert_Sim_alternative_step2}. 
\begin{algorithm} 
 \caption{Improved Step~\ref{alg_step_determine} of Algorithm~\ref{alg_Bert_Sim_structure}: Determine the subproblems to be solved. }\label{alg_Bert_Sim_alternative_step2}
 \begin{algorithmic}[1]
	\Require an edge set $E$ with cost interval lengths $\delta$, a value $\Gamma \leqq |E|$
	\Ensure an index set $L$ of subproblems to be solved in Algorithm~\ref{alg_Bert_Sim_structure}
  	 \State\label{alg_Bert_Sim_determine_begin} $l := \Gamma + 1$ \Comment{\emph{Lemma~\ref{lemma_alg_Bertsim_faster_1} (\ref{lemma_alg_Bertsim_faster_1_b}.,~\ref{lemma_alg_Bertsim_faster_1_c}.)}}
	 \State $L:= \{l\}$ 
	 \While {$l<|E|+1$}
	    \While {$l<|E|+1$ and $\delta_{e_l} = \delta_{e_{l+1}}$} $l := l+1$ \Comment{\emph{Lemma~\ref{lemma_alg_Bertsim_faster_1} (\ref{lemma_alg_Bertsim_faster_1_a}.)}}
	    \EndWhile
	    \If {$l<|E|+1$} $l:=l+1$
		\If {$l<|E|+1$} $l:=l+1$ \Comment{\emph{Lemma~\ref{lemma_alg_Bertsim_faster_1} (\ref{lemma_alg_Bertsim_faster_1_c}.)}}
		\EndIf
		\State $L:= L \cup \{l\}$
	    \EndIf
	  \EndWhile\label{alg_Bert_Sim_determine_end}
 \end{algorithmic}
 \end{algorithm}
 
Depending on the solutions that are found, while the algorithm is executed, we can further reduce the number of subproblems to be solved. We will refer to this enhancement as \emph{solution checking}.  
\begin{lemma}\label{lemma_alg_Bertsim_faster_2}
 Let $1\leqq \tilde{l} \leqq l \leqq|E|+1$ and let $q^{\tilde{l}}$ be an optimal solution for $\cSP(\tilde{l})$. If $q^{\tilde{l}}$ does not contain any of the elements $e_1,...,e_{l-1}$, then it is optimal for $\cSP(l)$.
\end{lemma}
\begin{proof}
 We can find a solution of $\cSP(l)$ by solving a problem with the deterministic costs given in (\ref{bertsim_sp_lenghts}).
 For these costs we have \begin{align*}
             &\tilde{l} \leqq l \Rightarrow \idelta_{{\tilde{l}}} \geqq \idelta_{l}  \Rightarrow \ce{e_j}^{\tilde{l}} \leqq \ce{e_j}^l &&\ \forall \ e_j: j < \tilde{l}, \\
             &j \leqq l \Rightarrow \delta_{e_j} \geqq \idelta_{l} \Rightarrow \ce{e_j}^{\tilde{l}} = \nome{e_j} \leqq \nome{e_j} + (\delta_{e_j} - \idelta_{l}) = \ce{e_j}^l   &&\ \forall \ e_j: \tilde{l} \leqq j < l, \\
             &\tilde{l} \leqq l \Rightarrow \ce{e_j}^{\tilde{l}} =\nome{e_j} =  \ce{e_j}^l  &&\ \forall \ e_j: j \geqq l.
            \end{align*}
If $q^{\tilde{l}}$ does not contain any element $e_j : j < l$, then \[ \sum_{e\in {q^{\tilde{l}}}} \ce{e}^l =  \sum_{e\in {q^{\tilde{l}}}} \ce{e}^{\tilde{l}}\ \leqq \sum_{e\in q} \ce{e}^{\tilde{l}} \leqq \sum_{e\in q} \ce{e}^l \ \forall \ q \in Q,\] hence, ${q^{\tilde{l}}}$ is optimal for $\cSP(l)$. 
\end{proof}

We can therefore replace Step~\ref{alg_step_solve} of the basic structure (Algorithm~\ref{alg_Bert_Sim_structure}) with Algorithm~\ref{alg_Bert_Sim_alternative_step3}.\\
\begin{algorithm}
\caption{Improved step~\ref{alg_step_solve} of Algorithm~\ref{alg_Bert_Sim_structure}: Solve subproblems (with solution checking).}\label{alg_Bert_Sim_alternative_step3}
 \begin{algorithmic}[1]
  \Require $I=(E, Q,\nom, \delta, \Gamma)$ with $E$ ordered w.r.t. $\delta_e$, $\idelta$, an index set $L$ of subproblems
  \Ensure a set of solutions $\{q^l: l\in L\}$
 \State $\tilde{l} := 0$
    \ForAll{$l\in L$ in increasing order}
	\If{$\tilde{l} = 0$ or $q^{\tilde{l}} $ contains any element in $\{e_1,...,e_{l-1}\}$}\label{alg_Bert_Sim_alternative_step3_if}
	  \State Find an optimal solution $q^l$ for $(\cSP(l))$.
	\Else 
%	  \State 
	\ $q^l:= q^{\tilde{l}}$
	\EndIf
	\State $\tilde{l} := l$
     \EndFor
 \end{algorithmic}
\end{algorithm} 

Lemma~\ref{lemma_alg_Bertsim_faster_2} does not contain any theoretical complexity result since, in the worst case, still ${\left\lceil \frac{|E|-\Gamma}{2} \right\rceil +1}$ subproblems are solved. Nevertheless, the results of our experiments in Section~\ref{section_experiments} show the practical use of this improvement.

\subsubsection{Extension to the multi-objective problem with objective-independent element order}\label{subsection_detsubproblems_objective_independen} 
In this section we adjust Algorithm~\ref{alg_Bert_Sim_structure} for multi-objective problems with the following property:
\begin{de}
  An instance $(E, Q,\nom, \delta, \Gamma)$ has \emph{objective independent element order} if 
  \begin{itemize}
    \item there exists an order of the elements, such that \[ \delta_{e_1,i} \geqq ... \geqq \delta_{e_{|E|},i} \ \forall \ i=1,...,k,\]
  \item and for all objective functions, the number of elements that may differ from the nominal value is the same, that is $\Gamma_1 = \Gamma_2 = ... = \Gamma_k$. In the following we use $\Gamma_1$ to denote the bound on the number of elements that may differ from the nominal value for each individual objective function.
  \end{itemize}
  
\end{de}

For multi-objective subproblems with objective independent element order, Algorithm $\ref{alg_Bert_Sim_structure}$ can be adjusted in the following way: \\
In step~\ref{alg_step_solve}, since $\Gamma, \delta_e, \nome{e}$ are vectors instead of scalars, the subproblems to be solved are the multi-objective problems
\[(\cMSP(l))\ \min_{q\in Q} \zBS^l(q) := \sum_{e \in q} \nome{e} + \Gamma \circ \idelta_{l} + \sum_{\substack{e_j \in q \\ j \leqq l}} (\delta_{e_j} - \idelta_{l}) \]
for $l=1,...,|E|+1$, with $\circ$ being the Schur (entry-wise) product and $\idelta_{|E|+1} = (0,...,0), \idelta_j := \delta_{e_j} \ \forall \ e_j \in E$. Because we solve multi-objective problems, we are looking for a complete set of efficient solutions for each subproblem instead of a single solution. Such a solution set can be found by solving a deterministic multi-objective problem. We denote the solution set, that we obtain for $(\cMSP(l))$, by $\ZBS^{l}$. 
\\ In Step~\ref{alg_step_compare}, every found solution $q$ whose objective vector $\zRC(q)$ is not dominated by the objective vector of any of the other solutions is robust efficient. We will refer to this special version of the \DSA{} for multi-objective instances with objective independent element order as \emph{objective independent \DSA{}} or \emph{\DSA{}-oi}.
\begin{algorithm} 
\caption{\DSA{} for multi-objective instances with objective independent element order (\DSA{}-oi)}\label{alg_Bert_Sim_sortable}
	\begin{algorithmic}[1]
	\Require an instance $I=(E, Q,\nom, \delta, \Gamma)$ of a multi-objective cardinality-constrained uncertain combinatorial optimization problem with objective independent element order
	\Ensure a complete set of robust efficient solutions for $I$
	 \State Sort $E$ w.r.t.~$\delta_e$ such that $\idelta_1 := \delta_{e_1} \geqq \idelta_2 := \delta_{e_2} \geqq ... \geqq \idelta_{{|E|}} \geqq \idelta_{|E|+1} := (0,...,0).$ 
	 \State Determine $L:= \{1,...,|E|+1\}$.
	 \State For all $l \in L$ find a complete set of efficient solutions $\ZBS^l$ for $(\cMSP(l))$.
	 \State Compare the objective vectors $\zRC(q)$ of all solutions in $\cup_{l \in L} \ZBS^l$.  The solutions with non-dominated objective vectors form a complete set of robust efficient solutions. 
	\end{algorithmic}
\end{algorithm}
 \begin{theo} \label{theo_alg_Bert_Sim_sortable}
 Algorithm~\ref{alg_Bert_Sim_sortable} finds a complete set of robust efficient solutions for multi-objective cardinality-constrained uncertain combinatorial optimization problems with objective independent element order.
\end{theo}
\begin{proof}
    First, we show that $\zBS^{l}$ never underestimates $\zRC$ for any objective. Further, we prove that for each feasible solution $q$ there is an $l \in \{\Gamma_1,...,|E|+1\} \subseteq L$ with $\zBS^{l}(q) = \zRC(q)$. We conclude that Algorithm~\ref{alg_Bert_Sim_sortable} finds a complete set of robust efficient solutions.\\
    For each $q \in Q$ and $l=1,...,|E|+1$ we show $ \zRC_i(q) \leqq \zBS^l_i(q) \ \forall \ i=1,...,k$. Let $\{ e_{a_1},\ldots, e_{a_h} \}$ be a set of $h$ elements in $q$ with the largest cost intervals, where $h =\min\{|q|,\Gamma_i \}$. Then,
	\begin{align*}
	  \zBS^l_{i}(q) &=&& \sum_{e\in q} \nomei{e}{i} + \Gamma_i \cdot \idelta_{l,i} + \sum_{\substack{e_j\in q\\ j\leqq l}} ( \delta_{e_j,i} - \idelta_{l,i})\\
	  &=&&  \sum_{e\in q} \nomei{e}{i} + \Gamma_i \cdot \idelta_{l,i} + \sum_{e\in q} \max\{0,\delta_{e,i} - \idelta_{l,i}\} &\text{ since } j\leqq l \Rightarrow \delta_{e_j} \geqq \idelta_l, j> l \Rightarrow \delta_{e_j} \leqq \idelta_l  \\
	  &\geqq&&  \sum_{e\in q} \nomei{e}{i} + \Gamma_i \cdot \idelta_{l,i} + \sum_{j=1}^h \max\{0,\delta_{e_{a_j},i} - \idelta_{l,i}\} & \text{ since } \{ e_{a_1},\ldots, e_{a_h} \} \subseteq q \\
	  &\geqq&& 	\sum_{e\in q} \nomei{e}{i} + \Gamma_i \cdot \idelta_{l,i} + \sum_{j=1}^h (\delta_{e_{a_j},i} - \idelta_{l,i})\\
	  &\geqq&& 	\sum_{e\in q} \nomei{e}{i} + \sum_{j=1}^h \idelta_{l,i} + \sum_{j=1}^h (\delta_{e_{a_j},i} - \idelta_{l,i}) = \zRC_i(q) & \text{ since } |\{ e_{a_1},\ldots, e_{a_h} \}| \leqq \Gamma_i.      
	\end{align*}
    We show now that there is an $l\in\{\Gamma_1,...,|E|+1\}$ with $\zBS^{l}(q) = \zRC(q)$: For any set $q \in Q$ with at least $\Gamma_1$ elements, let $e_{l}$ be the element with the $\Gamma_1$-th smallest index. Then the $\Gamma_1$ elements $\{e_j\in q: j \leqq l\}$ have the largest cost intervals in $q$ with respect to every objective. It follows for all $i=1,...,k$ that
    \begin{align*}
    \zBS^l_{i}(q) &=&& \sum_{e\in q} \nomei{e}{i} + \Gamma_i \cdot \idelta_{l,i} + \sum_{\substack{e_j\in q\\ j\leqq l}} ( \delta_{e_j,i} - \idelta_{l,i})\\
    &=&&  \sum_{e\in q} \nomei{e}{i} + \sum_{\substack{e_j\in q\\ j\leqq l}} \idelta_{l,i} + \sum_{\substack{e_j\in q\\ j\leqq l}} ( \delta_{e_j,i} - \idelta_{l,i}) = \zRC_i(q) & \text{ since } |\{e_j\in q: j \leqq l\}| =  \Gamma_i.
    \end{align*}
    For any set $q\in Q$ with less than $\Gamma_1$ elements, we have for all $i=1,...,k$
    \begin{align*}
    \zBS^{|E|+1}_{i}(q) = \sum_{e\in q} \nomei{e}{i} + \Gamma_i \cdot 0 + \sum_{\substack{e_j\in q}} (\delta_{e_j,i} - 0) = \zRC_i(q).
    \end{align*}
    We conclude: If $q$ is robust efficient, then $\zRC(q) = \zBS^l(q)$ for some $l \in L$ and there is no $q' \in Q$ with $\zRC(q') \leq \zRC(q)$. It follows that
    \begin{align*}
    \nexists q' \in Q: \zRC(q') \leqq \zRC(q) \overset{ \zRC(q') \leqq \zBS^l(q')}{\Rightarrow } \nexists q' \in Q:  \zBS^l(q') \leqq \zRC(q) = \zBS^l(q).
    \end{align*}
    Therefore, $q$ or an equivalent solution is found at least once in the algorithm. It follows that in Step~\ref{alg_step_compare} the objective vector of each found solution is compared to all non-dominated objective vectors, thus only robust efficient solutions remain. It follows that the output is a complete set of robust efficient solutions.
\end{proof}
  Now, we consider the enhancements proposed in Algorithms~\ref{alg_Bert_Sim_alternative_step2} and \ref{alg_Bert_Sim_alternative_step3}. The results of Lemma~\ref{lemma_alg_Bertsim_faster_1} remain valid, Step~\ref{alg_step_determine} can hence be implemented as in Algorithm~\ref{alg_Bert_Sim_alternative_step2}. 
\begin{lemma}\label{lemma_alg_Bertsim_faster_1_sortable}
    The number of subproblems to be solved by {Algorithm~\ref{alg_Bert_Sim_sortable}} can be reduced to ${\left\lceil \frac{|E|-\Gamma_1}{2} \right\rceil +1}$ in the same ways as in the single-objective case:
    \begin{enumerate}
    \item \label{lemma_alg_Bertsim_faster_1_sortable_a}  If there are several elements $e_l,...,e_{(l+h)}$ with the same interval length $\delta_{e_l} = ... = \delta_{e_{l+h}}$, only one of the subproblems $\cMSP(l), ..., \cMSP(l+h)$ needs to be solved. 
    \item \label{lemma_alg_Bertsim_faster_1_sortable_c} Only every second subproblem and $\cMSP(|E|+1)$ need to be solved.
    \item \label{lemma_alg_Bertsim_faster_1_sortable_b} It is sufficient to start with $\cMSP(\Gamma_1)$.
    \end{enumerate}
\end{lemma}
\begin{proof}\
    \begin{enumerate}
      \item From $\delta_{e_l} = ... = \delta_{e_{l+h}}$ follows directly $\zBS^l(q) = ... = \zBS^{l+h}(q)$ and therefore $\ZBS^l(q) = ... = \ZBS^{l+h}(q)$.
      \item For any $q \in Q$ with less than $\Gamma_1$ elements we have $\zRC(q) = \zBS^{|E|+1}(q)$. For any $q\in Q$ with at least $\Gamma_1$ elements %let $e_l$ be the element in $q$ with the $\Gamma_1$-th smallest index. 
      let $e_l$ be the element with the $\Gamma_1$-th smallest index in $q$. 
      From the proof of Theorem~\ref{theo_alg_Bert_Sim_sortable} we know that $\zRC(q) = \zBS^l(q)$. We further have
      \begin{align*}
	      \zBS^l(q) &=&& \sum_{e\in q} \nome{e} + \Gamma \circ \idelta_l + \sum_{\substack{e_j\in q\\j\leqq l}} (\delta_{e_j} - \idelta_{l})  + \Gamma \circ (\idelta_{l+1} - \idelta_{{l}}) + \Gamma \circ (\idelta_{l} - \idelta_{{l+1}}) \\% \hspace{1cm} \text{ as } |\{e_j\in q : j\leqq l\}| = \Gamma \\
	      &=&& \sum_{e\in q} \nome{e} + \Gamma \circ \idelta_{l+1} + \sum_{\substack{e_j\in q\\j\leqq l}} (\delta_{e_j} - \idelta_{{l+1}}) \hspace{1cm} \text{ because } |\{e_j \in q: j\leqq l\}|=\Gamma_i \ \forall \ i\\
	      &=&& \sum_{e\in q} \nome{e} + \Gamma \circ \idelta_{l+1} + \sum_{\substack{e_j\in q\\j\leqq l+1}} (\delta_{e_j} - \idelta_{{l+1}}) = \zBS^{l+1}(q).
      \end{align*}
      Therefore, if we only solve every second subproblem, we will still solve at least one subproblem $\cMSP(l')$ with $\zBS^{l'}(q) = \zRC(q)$ for each $q \in Q$ with at least $\Gamma_1$ elements. It follows, that we only need to solve every second subproblem and in addition the $(|E|+1)$-th subproblem. 
      \item In the proof of Theorem~\ref{theo_alg_Bert_Sim_sortable} we show that for every $q \in Q$ there is an $l \in {\{\Gamma_1,...,|E|+1\} }$ with $\zRC(q) = \zBS^l(q)$,  
      because none of the elements $e_1,...,e_{\Gamma_1-1}$ can be the element with the $\Gamma_1$-th smallest index in $q$. It follows, that we do not need to solve the problems $\cMSP(1),...,\cMSP(\Gamma_1-1)$ to find a complete set of robust efficient solutions. 
    \end{enumerate}
From statement~\ref{lemma_alg_Bertsim_faster_1_sortable_c} we know that at most $|E|+1 - (\Gamma_1 - 1)$ problems need to be solved. From statement~\ref{lemma_alg_Bertsim_faster_1_sortable_b} it follows that of these problems only the last one and every second of the other ones must be solved, this leads to at most \[ \left\lfloor \frac{|E| + 1 - (\Gamma_1 - 1) - 1}{2} \right\rfloor + 1 =\left\lfloor \frac{|E| - \Gamma_1 +1 }{2} \right\rfloor + 1= {\left\lceil \frac{|E|-\Gamma_1}{2} \right\rceil +1} \] subproblems.
\end{proof}
The result of Lemma~\ref{lemma_alg_Bertsim_faster_2} is valid for multi-objective problems with objective independent element order as well. However, to be able to skip the solving of problem $\cMSP(l)$ none of the sets in $\ZBS^{\tilde{l}}$ is allowed to contain any element $e_j$ with $j<l$. Therefore we replace Step~\ref{alg_step_solve}  with  Algorithm~\ref{alg_Bert_Sim_alternative_step3_mo1}.
\begin{algorithm}
\caption{Improved step~\ref{alg_step_solve} of Algorithm~\ref{alg_Bert_Sim_sortable}: Solve subproblems (with solution checking).}\label{alg_Bert_Sim_alternative_step3_mo1}
 \begin{algorithmic}[1]
 \Require $I=(E, Q,\nom, \delta, \Gamma)$ with $E$ ordered w.r.t. $\delta_i$, $\idelta$, an index set of subproblems $L$
\Ensure a set of solutions $\cup_{l \in L} \ZBS^l$ 
 \State $\tilde{l} := 0$
    \ForAll{$l\in L$ in increasing order}
	\If{$\tilde{l} = 0$ or any of the sets in $\ZBS^{\tilde{l}}$ contains any element in $\{e_1,...,e_{l-1}\}$}
	  \State Find a complete set of efficient solutions $\ZBS^l$ for $(\cMSP(l))$
	\Else \phantom{i}$\ZBS^l:= \ZBS^{\tilde{l}}$
	\EndIf
	\State $\tilde{l} := l$
     \EndFor
 \end{algorithmic}
\end{algorithm} 
\begin{lemma}\label{lemma_alg_Bertsim_faster_2_sortable}
 Let $1\leqq \tilde{l} \leqq l \leqq|E|+1$ and let $G^{\tilde{l}}$ be a complete set of efficient solutions for $\cMSP(\tilde{l})$. If none of the sets in $G^{\tilde{l}}$ contains any of the elements $e_1,...,e_{l-1}$, then $G^{\tilde{l}}$ is a complete set of efficient solutions for $\cMSP(l)$.
\end{lemma}
\begin{proof}
   A complete set of solutions for $\cMSP(l)$ can be found by solving a deterministic multi-objective problem with costs $\ce{e}^l:=(\cei{e}{1}^l, ..., \cei{e}{k}^l)$:
  \[ \cei{e_j}{i}^l := \begin{cases}
         \nomei{e_j}{i} + (\delta_{e_j,i} - \idelta_{l,i}) &\text{ for } j < l \\
         \nomei{e_j}{i} &\text{ for } j \geqq l.
        \end{cases}                                                                                                                                                                     
     \] 
     From the proof of Lemma~\ref{lemma_alg_Bertsim_faster_1} we know that $\cei{e_j}{i}^{\tilde{l}} \leqq \cei{e_j}{i}^l \ \forall \ e_j: j < l$ and $\cei{e_j}{i}^{\tilde{l}} = \cei{e_j}{i}^l \ \forall \ e_j: j \geqq l$. It follows that any $q\in \ZBS^{\tilde{l}}$ not containing any element in $\{e_1,...,e_{l-1}\}$ is also efficient w.r.t.~$\cost^l$. If none of the sets in $\ZBS^{\tilde{l}}$ contains any element in $\{e_1,...,e_{l-1}\}$, then for any $q'\notin \ZBS^{\tilde{l}}$ exists a $q \in \ZBS^{\tilde{l}}$ with 
     \[\sum_{e\in q} \ce{e}^l = \sum_{e\in q} \ce{e}^{\tilde{l}} \leqq \sum_{e\in q'} \ce{e}^{\tilde{l}} \leqq \sum_{e\in q'} \ce{e}^l \] 
     and $q'$ is either dominated w.r.t.~$\cost^l$ or has an equivalent solution in $ \ZBS^{\tilde{l}}$. Therefore, $\ZBS^{\tilde{l}}$ is a complete set of solutions for $\cMSP(l)$.
\end{proof}
\begin{cor}
 If we replace in Algorithm~\ref{alg_Bert_Sim_sortable} Step~\ref{alg_step_determine} with Algorithm~\ref{alg_Bert_Sim_alternative_step2} and Step~\ref{alg_step_solve} with Algorithm~\ref{alg_Bert_Sim_alternative_step3_mo1}, it finds a complete set of robust efficient solutions for multi-objective cardinality-constrained uncertain combinatorial optimization problems with objective independent element order, solving at most ${\left\lceil \frac{|E|-\Gamma_1}{2} \right\rceil +1}$ deterministic subproblems.
\end{cor}

\subsubsection{The Deterministic Subproblems Algorithm in the general multi-objective case}\label{subsection_detsubproblems_multi-objective}

In general, the sorting of the elements by interval lengths results in a different order for each objective. An element that has the $\Gamma$-th longest interval in $q$ for all objectives is not likely to exist. To ensure that the worst case vector of $q$ equals the objective vector of a subproblem, we have to iterate through all elements for each objective independently and consider all possible combinations.
\\
Let $E^i_j$ be a set of the $j$ elements with the largest intervals for the $i$-th objective, i.e., $|E^i_j| = j$ and $\delta_{e,i} \geqq \delta_{e',i} \ \forall \ e \in E^i_j, e' \in E\setminus E^i_j$ and let $\idelta_j^i := \min_{e\in E^i_j} \delta_{e,i}$. \\
 We define $\idelta_{|E|+1}^i:= 0 \ \forall \ i$. For each $l=(l_1,...,l_k) \in\{1,...,|E|+1\}\times ... \times\{1,...,|E|+1\}$  
%$l=(l_1,...,l_k): {l_i \in\{1,....,|E|+1\}} \ {\forall \ i=1,...,k}$ 
we define the problem
\begin{align*}
 (\cGMSP(l))\ \min_{q\in Q} \zBS^l(q) :=  \begin{pmatrix} \sum_{e \in q} \nomei{e}{1} + \Gamma{_1} \cdot \idelta_{l_1}^1 + \sum_{\substack{e \in q \cap  E^1_{l_1}}} (\delta_{e,1} - \idelta_{l_1}^1) \\
														\vdots \\
							\sum_{e \in q} \nomei{e}{k} + \Gamma{_k} \cdot \idelta_{l_k}^k + \sum_{\substack{e \in q \cap  E^k_{l_k}}} (\idelta_{e,k} - \idelta_{l_k}^k)
                                                                   \end{pmatrix}.
\end{align*}
As before, each of these $(|E|+1)^k$ problems can be solved as a deterministic multi-objective problem of the same kind.\\
Algorithm~\ref{alg_Bert_Sim_general} preserves the basic structure of \DSA: First, the elements are sorted w.r.t.~$\delta_{e,i}$ for each $i=1,...,k$. Instead of changing the indices, we store the set $E^i_j$ of the first $j$ elements for all $j=1,...,|E|$, because the order of the elements depends on the objective. Then the set $L$ is determined, which contains vectors instead of scalar values. For each element in $L$ the subproblem defined above is solved and their solutions are compared to obtain the robust efficient solutions.
\begin{algorithm} 
\caption{\DSA{} for general multi-objective instances}\label{alg_Bert_Sim_general}
	\begin{algorithmic}[1]
	\Require an instance $I=(E, Q,\nom, \delta, \Gamma)$ of a multi-objective cardinality-constrained uncertain combinatorial optimization problem 
	\Ensure a complete set of robust efficient solutions for $I$
	 \State For $ i := 1,...,k$: %\Statex 
	 Sort $E$ w.r.t.~$\delta_{e,i}$ descending and save the first $j$ elements in $E^i_j$  for $j=1,...,|E|$. Set $E^i_{|E|+1}:= E$. Set $\idelta_{j}^i:=\min_{e\in E^i_j} \delta_e \ \forall \ j=1,...,|E|$ and $\idelta_{|E|+1}^i:= 0$.
	 \State Determine $L=L_1 \times L_2 \times... \times L_k$: $L_i:= \{1,...,|E|+1\}\ \forall \ i=1,...,k$.
% 	 \ForAll{$l \in L$}
% 	    \State\label{alg_step_solve} Solve $(\cSP(l))$.
% 	 \EndFor
	 \State For all $l \in L$ find a complete set of efficient solutions $\ZBS^l$ for $(\cGMSP(l))$.
	 \State Compare the objective vectors $\zRC(q)$ of all solutions in $\cup_{l \in L} \ZBS^l $. The solutions with non-dominated objective vectors form a complete set of robust efficient solutions.
	\end{algorithmic}
\end{algorithm}

\begin{theo} \label{theo_alg_Bert_Sim_general}
 Algorithm~\ref{alg_Bert_Sim_general} finds a complete set of robust efficient solutions for multi-objective cardinality-constrained uncertain combinatorial optimization problems.
\end{theo}
\begin{proof} 
Analogously to the proof of Theorem~\ref{theo_alg_Bert_Sim_sortable} it can be shown that $\zRC(q) \leqq \zBS^l(q) \ \forall \ q\in Q$ and that for every $q\in Q$ and $i\in \{1,...,k\}$ there exists an $l$ with $l_i \in \{1,...,|E|+1\}$ and $\zRC_i(q) = \zBS^l_i(q)$. Since we consider every combination of the values of $l_1,...,l_k$ it follows that for every $q \in Q$ it exists a problem $\cGMSP(l)$ with $\zRC_i(q) = \zBS^l_i(q) \ \forall \ i = 1,...,k$. It follows that a complete set of robust efficient solutions is returned.
\end{proof}
As before, we can reduce the number of subproblems to be solved. The proof of Lemma~\ref{lemma_alg_Bertsim_faster_1_sortable} still holds for each objective independently. Therefore, we can apply the reduction on each objective  (Algorithm~\ref{alg_Bert_Sim_alternative_step2_mo2}) and obtain $L:= L_1\times L_2\times...\times L_k$ with $|L_i| = \left\lceil \frac{|E|-\Gamma_i}{2} \right\rceil +1$ .
 \begin{algorithm} 
 \caption{Improved Step~\ref{alg_step_determine} of Algorithm~\ref{alg_Bert_Sim_general}: Determine the subproblems to be solved. }\label{alg_Bert_Sim_alternative_step2_mo2}
 \begin{algorithmic}[1]
   \Require an edge set $E$ with cost interval lengths $\delta$, a $k$-dimensional vector $\Gamma$ with $\Gamma_i \leqq |E| \ \forall i$
  \Ensure an index set $L$ of subproblems to be solved in Algorithm~\ref{alg_Bert_Sim_general}
    \For{$i=1,...,k$}
  	 \State Determine $L_i$ as in Algorithm~\ref{alg_Bert_Sim_alternative_step2} with $\Gamma = \Gamma_i, \delta_e = \delta_{e,i}$. 
    \EndFor
 \end{algorithmic}
 \end{algorithm}
\\Here again, we can use solution checking, i.e., skip some additional subproblems, depending on the solutions found so far.

\begin{lemma}\label{lemma_alg_Bertsim_faster_2_general_a}
 Let $l,\tilde{l} \in \Z^k$ be given with  $l\le \tilde{l}$ and let $I$ be the set of indices $i$ with $l_i<l_{\tilde{i}}$. Let $\ZBS^{\tilde{l}}$ be a complete set of efficient solutions for $\cGMSP(\tilde{l})$. If for all $i$ for which $l_i<l_{\tilde{i}}$, none of the sets in $\ZBS^{\tilde{l}}$ contains an element in $\cup_{i \in I} E^i_{l_i}$, then $\ZBS^{\tilde{l}}$ is a complete set of efficient solutions for $\cGMSP(l)$. 
\end{lemma}

\begin{proof}
 Since $\Gamma_i \cdot \idelta^{l_i}_i$ are solution independent constants, the minimization problem to be solved is a deterministic multi-objective problem with costs $\ce{e}^l = (\cei{e}{1}^l,...,\cei{e}{k}^l):$
 \begin{align*}
  \cei{e}{i}^l := \begin{cases}
         \nomei{e}{i} + (\delta_{e,i} - \idelta^i_{l_i} %\delta_{e_{l_i},i}
         )&\text{ for }e \in E^i_{l_i} \\
         \nomei{e}{i} &\text{ else}.
        \end{cases}
 \end{align*}
 Therefore, 
 \begin{align*}
  &\cei{e}{i}^l = \cei{e}{i}^{\tilde{l}} \ \forall \ i \text{ with } l_i=\tilde{l}_i, \ \forall \ e \in E \\
  &\cei{e}{i}^l = \cei{e}{i}^{\tilde{l}} \ \forall \ i \text{ with } l_i<\tilde{l}_i, \ \forall \ e \in E\setminus E^i_{l_i} \\
  &\cei{e}{i}^{\tilde{l}} \leqq \cei{e}{i}^l \ \forall \ i, \  \forall \ e \in E.
  \end{align*}
  Hence, for all objective functions $i$ we have $\cei{e}{i}^l = \cei{e}{i}^{\tilde{l}}$ for all elements that are contained in any set in $\ZBS^{\tilde{l}}$, and $\cei{e}{i}^{\tilde{l}} \leqq \cei{e}{i}^l$ for all elements that are not contained in a set in $\ZBS^{\tilde{l}}$.   Analogously to the proof of Lemma~\ref{lemma_alg_Bertsim_faster_2_sortable}, it follows that $\ZBS^{\tilde{l}}$ is a complete set of efficient solutions for $\cGMSP(l)$ if no solution in $\ZBS^{\tilde{l}}$ contains  any element in $\cup_{i \in I} E^i_{l_i}$. \\
\end{proof}

A fast way to use this result is to replace Step~\ref{alg_step_solve} of Algorithm~\ref{alg_Bert_Sim_general} with Algorithm~\ref{alg_Bert_Sim_alternative_step3_mo2}.

 \begin{algorithm}
\caption{Improved step~\ref{alg_step_solve} of Algorithm~\ref{alg_Bert_Sim_general}: Solve subproblems (with solution checking).%: For all $l\in L$ solve $(\cGMSP(l))$ if necessary.
}\label{alg_Bert_Sim_alternative_step3_mo2}
 \begin{algorithmic}[1]
  \Require an instance $I=(E, Q,\nom, \delta, \Gamma)$, $\idelta$, edge sets $E^i_{j} \ \forall \ i,j$, an index set $L$ of subproblems
  \Ensure a set of solutions $\cup_{l \in L} \ZBS^l$ 
 \State $\tilde{l}^1 := (0,...,0) $%\ \forall \ i=1,...,k$
 \State $h := 1$
    \ForAll{$l_1\in L_1$ in increasing order}
      \ForAll{$l_2\in L_2$ in increasing order}
      \State ...
	\ForAll{$l_k\in L_k$ in increasing order}
	    \State $l:=(l_1,...,l_k)$
	    \If{$\tilde{l}^h = (0,...,0)$ or any of the sets in $\ZBS^{\tilde{l}^h}$ contains any element in $E^h_{l_h}$}\label{alg_Bert_Sim_general_check}
	      \State Find a complete set of efficient solutions $\ZBS^l$ for $(\cGMSP(l))$.
	    \Else
	      \ $\ZBS^{l} := \ZBS^{\tilde{l}^h}$
	    \EndIf
	    \For{$i=h,...,k$} \label{alg_Bert_Sim_general_tilde2_for}
		\State ${\tilde{l}^i} := l$ \label{alg_Bert_Sim_general_tilde2}
	    \EndFor
	    \State $h:= k$\label{alg_Bert_Sim_general_h}
	\EndFor
	\State ...
	\State $h:=2$
      \EndFor
      \State $h:=1$
     \EndFor
 \end{algorithmic}
\end{algorithm} 

\begin{lemma}\label{lemma_alg_Bertsim_faster_2_general_b}
 In Line~\ref{alg_Bert_Sim_general_check} of Algorithm~\ref{alg_Bert_Sim_alternative_step3_mo2}, if $\tilde{l}^h \neq (0,...,0)$, then $1 \leqq \tilde{l}_h < l_h \leqq |E|+1$  and $\tilde{l}_i = l_i \ \forall \ i=1,...,k, i \neq h$.
\end{lemma}
\begin{proof}
  For every $i=1,...,k$ let $l_i^1:= \min_{l_i \in L_i} l_i$ be the minimal element in $L_i$.
  We use the following observations:
  \begin{enumerate}
  \item\label{proof_lemma_alg_Bertsim_faster_2_general_obs_2} Because $h=1$ and $\tilde{l}^1 = (0,...,0)$ in the first iteration, the first subproblem is solved and then $\tilde{l}^i$ is set to $\tilde{l}^i:=(l_1^1,l_2^1,...,l_k^1)$ for all $i=1,...,k$ in Line~\ref{alg_Bert_Sim_general_tilde2}. Thereafter, $\tilde{l}^i$ is changed in Line~\ref{alg_Bert_Sim_general_tilde2} if and only if $h\leqq i$.
  \item\label{proof_lemma_alg_Bertsim_faster_2_general_obs_3} The value of $h$ is changed to $\hat{h}$ whenever one iteration of the for-loop changing $l_{\hat{h}}$ is finished. Then $l_{i}= l^1_{i} \ \forall \ i > \hat{h}$ during the next execution of Lines~\ref{alg_Bert_Sim_general_check} to~\ref{alg_Bert_Sim_general_h}. 
  \item\label{proof_lemma_alg_Bertsim_faster_2_general_obs_4} In Line~\ref{alg_Bert_Sim_general_check} \[ h = \min \{ i \in \{1,...,k\}: l_i \text{ has changed since the last execution of Line~\ref{alg_Bert_Sim_general_check}} \}.\] 
  \end{enumerate}
  We consider the state of the algorithm during any iteration of Line~\ref{alg_Bert_Sim_general_check} and show $\tilde{l}^h_h < l_h$ and $\tilde{l}^h_i = l_i \ \forall \ i \neq h$. Let $\hat{h}$ denote the value of $h$ at this moment.
  \begin{itemize}
   \item $i > \hat{h}$:  When $\tilde{l}^{\hat{h}}$ was changed last, $h \leqq \hat{h} < i$ hold (\ref{proof_lemma_alg_Bertsim_faster_2_general_obs_2}.), so $\tilde{l}_i^{\hat{h}}$ was set to $l_i^1$ (\ref{proof_lemma_alg_Bertsim_faster_2_general_obs_3}.). It follows  $\tilde{l}_i^{\hat{h}} = l_i^1 \overset{\ref{proof_lemma_alg_Bertsim_faster_2_general_obs_3}.}{=} l_i$.
   \item $i< \hat{h}$: When $\tilde{l}^{\hat{h}}$ was changed, either $h < i$ or $h \geqq i$ hold. If it was $h < i$ then $\tilde{l}_i^{\hat{h}}$ was set to $l_i^1$ (\ref{proof_lemma_alg_Bertsim_faster_2_general_obs_3}.) and it follows  $\tilde{l}_i^{\hat{h}} = l_i^1 \overset{\ref{proof_lemma_alg_Bertsim_faster_2_general_obs_3}.}{=} l_i$. \\ If it was $h \geqq i$, then $l_i$ was not changed since then, otherwise $\tilde{l}^{\hat{h}}$ would have been changed again, because of $i<\hat{h}$ (\ref{proof_lemma_alg_Bertsim_faster_2_general_obs_2}.). It follows $l_i = \tilde{l}^{\hat{h}}_i$.
   \item $i = \hat{h}$: 
  We show first that $l_{\hat{h}}$ changed at most once since the last change of $\tilde{l}^{\hat{h}}$. During the first execution of Line~\ref{alg_Bert_Sim_general_check} after the first change of $l_{\hat{h}}$ it holds $h\geqq \hat{h}$ (\ref{proof_lemma_alg_Bertsim_faster_2_general_obs_4}.). So  $\tilde{l}^{\hat{h}}$ is changed again in Line~\ref{alg_Bert_Sim_general_tilde2}, before $l_{\hat{h}}$ could be changed a second time. \\
  From (\ref{proof_lemma_alg_Bertsim_faster_2_general_obs_4}.) follows, that $l_{\hat{h}}$ has changed since the last execution of Line~\ref{alg_Bert_Sim_general_check}, but $l_{\hat{h}-1}$ hasn't. Therefore, the for-loop changing $l_{\hat{h}-1}$ has not been finished. Hence, $l_{\hat{h}}$ can not have been set to $l^1_{\hat{h}}$ again, but was increased. This was the only change of $l_{\hat{h}}$ since $\tilde{l}^{\hat{h}}$ was set to its current value. It follows $\tilde{l}^{\hat{h}}_{\hat{h}} < l_{\hat{h}}$.
  \end{itemize}
\end{proof}
\begin{cor}\label{cor_Bert_Sim_general}
 If we replace in Algorithm~\ref{alg_Bert_Sim_general} Step~\ref{alg_step_determine} with Algorithm~\ref{alg_Bert_Sim_alternative_step2} and 
 Step~\ref{alg_step_solve} with Algorithm~\ref{alg_Bert_Sim_alternative_step3_mo2}, it still finds a complete set of robust efficient solutions for general instances of multi-objective cardinality-constrained uncertain combinatorial optimization problems. During its execution at most $\prod_{i=1}^k \left( \left\lceil \frac{|E|-\Gamma_i}{2} \right\rceil +1 \right)$ deterministic subproblems have to be solved.
\end{cor}
The number of subproblems to be solved for general instances is hence a lot higher than for instances with objective independent element order. But if there exists a common element order for only a subset of the objectives, we can already reduce the number of subproblems significantly:
\begin{de}\label{def_partly_oieo}
 An instance $(E, Q,\nom,\delta,\Gamma)$ has \emph{partial objective independent element order} if there is a subset $\{i_1,...,i_r\} \subset \{1,...,k\}$ with
  \begin{itemize}
  \item $\Gamma_{i_1} = \Gamma_{i_2} = ... = \Gamma_{i_r}$ and
  \item there exists an order of the elements, such that \[ \delta_{e_1,i} \geqq ... \geqq \delta_{e_{|E|},i} \ \forall \ i \in \{i_1,...,i_r\}.\]
  \end{itemize}
\end{de}
\begin{lemma}
 Let an instance $(E, Q,\nom,\delta,\Gamma)$ with partial objective independent element order be given and let $\{i_1,...,i_r\}$ be the subset from definition~\ref{def_partly_oieo}. 
 Then the nested for-loops changing $l_{i_1},...,l_{i_r}$ in Algorithm~\ref{alg_Bert_Sim_alternative_step3_mo2} can be replaced by a single for-loop. The number of solved deterministic subproblems in Algorithm~\ref{alg_Bert_Sim_general} with Algorithm~\ref{alg_Bert_Sim_alternative_step2} as Step~\ref{alg_step_determine} and Algorithm~\ref{alg_Bert_Sim_alternative_step3_mo2} as Step 3 is then less or equal to
\[ \left(\left\lceil \frac{|E|-\Gamma_{i_1}}{2} \right\rceil +1 \right)  \cdot \prod_{i\in \{1,...,k\} \setminus \{i_1,...,i_r\}} \left( \left\lceil \frac{|E|-\Gamma_i}{2} \right\rceil +1 \right).\]
\end{lemma}
\begin{proof}
 This follows directly from the proofs of Theorem~\ref{theo_alg_Bert_Sim_sortable} and Lemma~\ref{lemma_alg_Bertsim_faster_1_sortable}.
\end{proof}

\subsection{\La{}  }\label{section_la}

\DSA{} is especially useful for problems with high $\Gamma_i$, because fewer subproblems have to be solved for higher values of $\Gamma_i$. For small $\Gamma_i$ the method described in \ref{subsection_la} might be preferred. Its idea is to transfer the multi-objective uncertain combinatorial optimization problem with $k$ objectives 
into a deterministic combinatorial optimization problem of the same kind 
with $\sum_{i=1}^{k}(\Gamma_i + 1)$ objective functions, some of which are 
bottleneck functions %($\min \kmaxs{j}{{e \in q}} c_e$) 
instead of sum functions. %($\min \sum_{e \in q} c_e$). 
The concept is particularly useful if an efficient algorithm for solving the deterministic multi-objective problem with sum and bottleneck functions is available. 
As an example we present such an algorithm for the shortest path problem in Section~\ref{subsection_labeling_sp}.

\subsubsection{\La{}  for cardinality-constrained uncertain combinatorial optimization problems}\label{subsection_la}
We first consider the single-objective uncertain problem $\left( \min_{q \in Q} z(q,\xi), \xi \in \cU \right)$.  
Its minmax robust counterpart is \begin{align} \label{problem_minmax_onecrit} \min_{q \in Q} \left( \zRC(q) := \max_{\xi \in \cU} z(q,\xi) \right).\end{align} 
\begin{de}
 For a set $A\subseteq \R$ let $\kmax{j} (A)$ denote the $j$-greatest element of $A$.\\
 For a set $q \subseteq E$ let $\kmaxs{j}{{e \in q}} \delta_e $ := $\kmax{j}(\{ \delta_e: e \in q \} )  $  denote the $j$-highest value $\delta_e$ which appears in $q$. 
If $q$ has less than $j$ elements we define \emph{$\kmaxs{j}{{e \in q}} \delta_e := 0 $}
\end{de}

\begin{theo}\label{theo_gamma_labeling_problem}
 Every optimal solution for (\ref{problem_minmax_onecrit}) is an efficient solution for the deterministic multi-objective problem 
\begin{align} \label{problem_gamma_label_onecrit} 
         \min\limits_{q\in Q} \zLA(q) := \begin{pmatrix}
                    \sum_{e\in q} \nome{e}\\
                    \max_{e\in q} \delta_{e} \\
                    \kmaxs{2}{{e\in q}} \delta_{e} \\
                    \vdots \\
                    \kmaxs{\Gamma}{{e \in q}} \delta_{e}
                   \end{pmatrix} .                                                                                                                                                                                                                                                                                                                                                                                                                                                                                                                                                                                                                                                                                                                                                                                               
\end{align}
\end{theo}
\begin{proof}
Let $q$ be an optimal solution for Problem~(\ref{problem_minmax_onecrit}). Assume that $q$ is not efficient for Problem~(\ref{problem_gamma_label_onecrit}). Then there exists a solution $q' \in Q$ that dominates $q$ and it follows
\begin{align*} 
    &\sum_{e\in q'} \nome{e} \leqq \sum_{e\in q} \nome{e} \text{ and } \underset{e\in q'}{\kmax{j}} \delta_{e} \leqq \underset{e\in q}{\kmax{j}} \delta_{e} \ \forall \ j=1,...,\Gamma  \text{, with at least one inequality}
    \\  
    \Rightarrow \ &
    \zRC(q') 
    {= }\sum_{e\in q'} \nome{e} +
    \sum_{j=1}^{\Gamma} \underset{e\in q'}{\kmax{j}}\delta_{e} <
    \sum_{e\in q} \nome{e} + 
    \sum_{j=1}^{\Gamma} \underset{e\in q}{\kmax{j}} \delta_{e}=\zRC(q),
\end{align*}
because the worst case scenario for any feasible set is a scenario where the cost of its $\Gamma$ elements with the largest cost intervals take their maximal values. This contradicts $q$ being optimal for (\ref{problem_minmax_onecrit}).
\end{proof}
The reverse of Theorem~\ref{theo_gamma_labeling_problem} does not hold: There exist efficient solutions for (\ref{problem_gamma_label_onecrit}), which are not optimal for (\ref{problem_minmax_onecrit}), as the following example shows.
  \begin{example}\label{example_bottleneck}
Let $G$ be a graph that consists of two disjoint %disjunct
paths $q, q'$ from $s$ to $t$ with three edges each. Let the cost interval of all edges in $q$ be $[1,1]$ and of all edges in $q'$ be $[0,1]$ and let $\Gamma = 2$. Then both paths are efficient solutions for Problem~(\ref{problem_gamma_label_onecrit}), because
\[\zLA(q)=(3,0,0)\nleq(0,1,1) = \zLA(q') \text{ and }\zLA(q') = (0,1,1)\nleq(3,0,0) = \zLA(q).\]
But only $q'$ is robust efficient, because
\[\zRC(q')=2<3=\zRC(q).\]
 \end{example}

\begin{lemma}\label{lemma_gamma_labeling_problem_2}
 A complete set of efficient solutions for Problem~(\ref{problem_gamma_label_onecrit}) contains at least one optimal solution for Problem~(\ref{problem_minmax_onecrit}).
\end{lemma}
\begin{proof}
 Let $Q'\subseteq Q$ be a complete set of efficient solutions for (\ref{problem_gamma_label_onecrit}). Assume, that (\ref{problem_minmax_onecrit}) has an optimal solution $q$ that is not contained in $Q'$. According to Lemma~\ref{theo_gamma_labeling_problem}, $q$ is an efficient solution for Problem (\ref{problem_gamma_label_onecrit}), so $Q'$ contains a solution $q'$ with 
		    \[\begin{pmatrix}
                    \sum_{e\in q} \nome{e}\\
                    \max_{e\in q} \delta_{e} \\
                    \kmaxs{2}{{e\in q}} \delta_{e} \\
                    \vdots \\
                    \kmaxs{\Gamma}{{e \in q}} \delta_{e}
                   \end{pmatrix} =
                   \begin{pmatrix}
                    \sum_{e\in q'} \nome{e}\\
                    \max_{e\in q'} \delta_{e} \\
                    \kmaxs{2}{{e\in q'}} \delta_{e} \\
                    \vdots \\
                    \kmaxs{\Gamma}{{e \in q'}} \delta_{e}
                   \end{pmatrix}
                   \Rightarrow \zRC(q) = \sum_{e\in q} \nome{e} + 
    \sum_{j=1}^{\Gamma} \underset{e\in q}{\kmax{j}} \delta_{e} = \zRC(q') \]
    and $q'$ is optimal for (\ref{problem_minmax_onecrit}).
\end{proof}
Now, we transfer this approach to the multi-objective case. For a problem with $k$ objectives, we construct a deterministic problem with $\sum_{i=1}^{k}(\Gamma_i + 1)$ objectives.   
\begin{theo}\label{theo_gamma_labeling_problem_mo}
 Every efficient solution for the multi-objective robust counterpart
 \begin{align}\label{problem_minmax_multicrit} \min_{q \in Q} \zRC(q) = \begin{pmatrix}
                                                               \max_{\xi \in \cU} z_1(q,\xi) \\
                                                               \vdots \\
                                                               \max_{\xi \in \cU} z_k(q,\xi) 
                                                              \end{pmatrix}
\end{align} 
 is an efficient solution for the deterministic multi-objective problem
 \begin{align}\label{problem_gamma_label_multicrit} 
         \min\limits_{q\in Q} \zLA(q) :=
         \begin{pmatrix}
                    \sum_{e\in q} \nomei{e}{1}\\
                    \max_{e\in q} \delta_{e,1} \\
                    \kmaxs{2}{{e\in q}} \delta_{e,1} \\
                    \vdots \\
                    \kmaxs{\Gamma_1}{{e \in q}} \delta_{e,1} \\
                    \sum_{e\in q} \nomei{e}{2}\\
                     \max_{e\in q} \delta_{e,2} \\
                    \vdots \\
                    \kmaxs{\Gamma_k}{{e \in q}} \delta_{e,k} 
       \end{pmatrix}  . 
\end{align}
A complete set of solutions for (\ref{problem_gamma_label_multicrit}) contains a complete set of solutions for (\ref{problem_minmax_multicrit}).
\end{theo}
\begin{proof}
Let $q$ be an efficient solution for Problem (\ref{problem_minmax_multicrit}).
Assume that $q$ is not efficient for Problem (\ref{problem_gamma_label_multicrit}). Analogously to the proof of Lemma $\ref{theo_gamma_labeling_problem}$, there is a path $q' \in Q$ dominating $q$ and it follows that
$\zRC_i(q') < \zRC_i(q)$ for at least one $i \in \{1,...,k\}$,
which contradicts $q$ being efficient for (\ref{problem_minmax_multicrit}). \\
Assume now, that $q \notin Q'$ with $Q'$ being a complete set of efficient solutions for Problem (\ref{problem_gamma_label_multicrit}). Since $q$ is efficient for Problem (\ref{problem_gamma_label_multicrit}), there is a solution $q'\in Q'$ equivalent to $q$ w.r.t.~the objective function of Problem (\ref{problem_gamma_label_multicrit}) and it follows
$\zRC(q) = \zRC(q')$ analogously to the proof of Lemma~\ref{lemma_gamma_labeling_problem_2}. 
\end{proof}
With an algorithm solving the deterministic Problem~(\ref{problem_gamma_label_multicrit}) and a method to filter the obtained solutions we can now find a complete set of robust efficient solutions for the uncertain problem. In the case of a single-objective uncertain problem, an algorithm to solve Problem~\eqref{problem_gamma_label_onecrit} was introduced in \cite{gorski2012generalized}.

\subsubsection{Label setting algorithm (\Lab{}) for the cardinality-constrained uncertain shortest path problem}\label{subsection_labeling_sp}

In this section, we show how to apply the \la{} to the cardinality-constrained uncertain shortest path problem. We propose an adjustment of standard multi-objective labeling algorithms (label setting or label correcting) to find a complete set of robust efficient solutions. \\ 
Let the cardinality-constrained uncertain shortest path problem be defined as in Section \ref{subsection_shortest_path}, i.e., $E$ is the edge set of a graph and $Q$ the set of simple paths from a given start node $s$ to a given end node $t$. For simplicity we consider the single-objective uncertain shortest path problem, i.e., we show how to solve Problem~(\ref{problem_gamma_label_onecrit}), but the adjustments can be used for Problem~(\ref{problem_gamma_label_multicrit}) in the same way. Additionally we assume non-negative edge costs ($\nome{e} \geqq 0 \ \forall \ e \in E$) and adjust a label setting algorithm as an example.
\\ We first recall the definition of a label, which is used in common multi-objective labeling algorithms. A label $l=(z,v',l')$ at a node $v$ consists of 
\begin{itemize}
 \item a cost vector $z$, here $z=(z_0,...,z_{\Gamma})$, 
 \item a predecessor node $v'$, and
 \item a predecessor label $l'$. 
 \end{itemize}
Every label at a node $v \neq s$ with predecessor node $v'$ \emph{represents} a path $q$ from $s$ to $v$ whose last edge is $(v',v)$.  That means that its cost equals the cost of $q$ and its predecessor label $l'$ represents the subpath of $q$ from $s$ to $v'$. We assume here, that no parallel edges exist, such that $v$ and $v'$ uniquely define an edge $(v',v)$. If parallel edges have to be considered, the respective edge can be contained in the label as well. The labels are constructed iteratively from existing labels at the predecessor nodes and can at any time be either \emph{temporary} or \emph{permanent}. \\
Algorithm~\ref{alg_labeling_structure} is a label setting algorithm for solving a shortest path problem of type~(\ref{problem_gamma_label_onecrit}). It is based on the label setting algorithm of Martins for multi-objective shortest path problems \cite{martins1984multicriteria}, but we make the following adjustments: 
\begin{enumerate}
 \item In Step~\ref{alg_labeling_structure_choose_label} a label must be chosen whose cost is not dominated by the cost of any other temporary label. In \cite{martins1984multicriteria} the lexicographically smallest label is chosen. Based on \cite{iori2010aggregate}, we choose the label with the smallest aggregate function $\sum_{i=0,...,\Gamma} z_i$ instead. 
 \item In multi-objective label setting algorithms with only sum functions (as in \cite{martins1984multicriteria}) a new label $l=(z,v',l')$ at $v$ is created by adding the cost $z'$ of the predecessor label $l'$ to the edge cost. For min-max functions the (entry-wise) maximum of the edge cost and the predecessor label's cost is taken (see \cite{gandibleux2006martins}). To solve Problem~(\ref{problem_gamma_label_onecrit}) we need a new way to construct the labels: For the sum objective function, we add the nominal costs $\nome{e}$ of the edge $e:=(v',v)$ to the corresponding predecessor cost entry $z_0$. For the  $\kmax{j}$ objective functions, we compare the interval length $\delta_{e}$ of $e$ to each of the $\Gamma$ longest interval lengths so far $z'_1,...,z'_{\Gamma}$ and insert it at the right position (see Algorithm~\ref{alg_labeling_create_label}). We will use the following notation: $z:= z'\oplus (\nome{e}, \delta_{e})$.
 \item In \cite{martins1984multicriteria} a newly created label is only deleted if it is dominated by a label at the same node. We delete the new label even if another label with equal cost exists at the same node, because we are only looking for a complete set of efficient solutions. This is also the reason why we do not need to consider \emph{hidden} labels, which were introduced by \cite{gandibleux2006martins} for problems with bottleneck functions. Since new labels with the same cost as existing labels are immediately deleted, Algorithm~\ref{alg_labeling_structure} works even without the assumption that no cycles of cost $(0,...,0)$ exist.
\end{enumerate}
\begin{algorithm} 
\caption{Label setting algorithm to solve a shortest path problem of type~\eqref{problem_gamma_label_onecrit}}\label{alg_labeling_structure}
	\begin{algorithmic}[1]
	  \Require an instance $I=(E,Q,\nom,\delta,\Gamma)$ of a multi-objective shortest path problem of type \eqref{problem_gamma_label_onecrit}
	  \Ensure permanent labels at $t$, representing a complete set of efficient solutions for $I$
	 \State Create a temporary label $l_0$ with cost (0,...,0) at node $s$.\label{alg_labeling_structure_create_first_label}
	 \While{there exists at least one temporary label}
	    \State Select a temporary label $l'$ (at any node $v'$) with minimal aggregate cost $\sum_{i=0,...,\Gamma} z'_i$ and make it permanent.\label{alg_labeling_structure_choose_label}% 
	    \ForAll{outgoing edges $(v',v)$ of $v'$}
	    \State Create a new temporary label $l$ at $v$ by Algorithm \ref{alg_labeling_create_label}.\label{alg_labeling_structure_create_label} 
	    \If {the cost of $l$ is dominated by or equal to the cost of another label at $v$} \State Delete $l$. \label{alg_labeling_structure_dominated}
	    \ElsIf{$l$ dominates any temporary labels at $v$} \State Delete these labels.
	    \EndIf
	    \EndFor
	 \EndWhile
	\end{algorithmic}
\end{algorithm}

\begin{algorithm}
\caption{Step \ref{alg_labeling_structure_create_label} of Algorithm~\ref{alg_labeling_structure}: Create a new temporary label.}\label{alg_labeling_create_label}
 \begin{algorithmic}[1]
  \Require an edge $(v',v)$, a label $l'$ with cost $z'$ at a node $v'$ 
  \Ensure a new label $l$ at $v$ with predecessor label $l'$ 
        \State $z_0 := z'_0 + \nome{(v',v)}$
	\State $i := 1$
	\While{$i \leqq \Gamma$}
	    \If{$\delta_{(v',v)} > z'_i$}
	      \State $z_i := \delta_{(v',v)}$ 
	      \For{ $j := i+1, ..., \Gamma$}
	      $z_j := z'_{j-1}$
	      \EndFor
	      \State $i := \Gamma + 1$
	    \Else \State $z_i := z'_i$ \State $i := i+1$  
	    \EndIf
	\EndWhile
	\State Create the temporary label $l:=((z_0,...,z_{\Gamma}),v',l')$ at node $v$.
 \end{algorithmic}
\end{algorithm}
 
\begin{lemma}\label{lemma_label_construction}
 In Algorithm~\ref{alg_labeling_structure} for every label $l=(z,v',l')$ at a node $v$ there exists a path $q$ from $s$ to $v$ with
 $ z = \zLA(q)$. 

\end{lemma}
\begin{proof}
 We show the statement by induction: \\ The first label has cost $(0,...,0)$ and represents the path only consisting of node $s$. \\ Let $z'=(z'_0,...,z'_{\Gamma})$ be the cost of the predecessor label $l'$ and assume that $z'$ equals the cost of a path $q'$ from $s$ to $v'$. Then we have
\begin{align*}
 z_0 = z'_0 + \nome{(v',v)} = \sum_{e\in q'} \nome{e}  + \nome{(v',v)} = \sum_{e\in q'\cup (v',v)} \nome{e}.
\end{align*}
For the other objectives we distinguish two cases:
\begin{itemize}
 \item Case 1: $\delta_{(v',v)} \leqq z'_i \ \forall i = 1,...,\Gamma$. In this case the $\Gamma$ edges $e$ with biggest intervals $\delta_e$ of $q'$ and $q'\cup (v',v)$ are the same and ${z}_i = z'_i$ for all objectives. Therefore, $({z}_0,...,{z}_\Gamma) = \zLA(q\cup (v',v))$.
 \item Case 2: Either $\delta_{(v',v)} > z'_i$ for $i=1$  or $\ \exists i \in \{2,...,\Gamma\}$ with $z'_{i-1} \geqq \delta_{(v',v)} > z'_i$. Then 
 \begin{align*}
  &\forall \ j< i :  \ {z}_j = z'_j \text{ and } \kmaxs{j}{{e\in q'}} \delta_{e} = \kmaxs{j}{{e\in q' \cup (v',v)}} \delta_{e} \\
  &\text{for } \ j= i :  \ {z}_j = \delta_{(v',v)} =  \kmaxs{j}{{e\in q'\cup (v',v)}} \delta_{e}  \\
  &\forall \ j:  \Gamma \geqq j> i :  \ {z}_j = z'_{j-1} =  \kmaxs{j}{{e\in q\cup (v',v)}} \delta_{e}
 \end{align*}
\end{itemize}
It follows $({z}_0,...,{z}_\Gamma) = \zLA(q' \cup (v',v))$.
\end{proof}
 
 In the deterministic case with only sum functions, subpaths of efficient paths are efficient as well, which plays an important role in the proof of Martin's algorithm. If some of the objective functions are bottleneck functions, this property does not hold any more \cite{gandibleux2006martins}. In our case, since we only look for a complete set of efficient solutions, the following weaker property is sufficient (this was observed but not proven in \cite{iori2010aggregate}).
 
 \begin{lemma}\label{lemma_subpaths_of_gamma_labeling}
 Let $q$ from $s$ to $t$ be an efficient path with respect to $\zLA$ and $v,w$ two nodes on $q$ ($v$ before $w$). %costs consisting of bottleneck and sum functions. 
 Then either $q_{v,w}$ is an efficient path from $v$ to $w$ or there exists an efficient path $p$ such that $q' :=q_{s,v} \cup p \cup q_{w,t}$ is equivalent to $q$.
\end{lemma}
\begin{proof}
Assume that $q_{v,w}$ is not efficient. Then there exists an efficient path $p$ from $v$ to $w$ that dominates $q_{v,w}$. We have 
\begin{align*}
 \sum_{e\in q'} \nome{e} =\sum_{e\in q_{s,v}} \nome{e} + \sum_{e\in p} \nome{e}+\sum_{e\in q_{w,t}} \nome{e} \leqq \sum_{q_{s,v}} \nome{e} +\sum_{e\in q_{v,w}} \nome{e}+\sum_{e\in q_{w,t}} \nome{e} = \sum_{e\in q} \nome{e}
\end{align*} and for $i=1,...,\Gamma$ it follows from $\kmaxs{j}{{e \in p}} \delta_{e} \leqq \kmaxs{j}{{e \in q_{v,w}}} \delta_{e}$ that $\forall \ j \leqq i$
\begin{align*}
 &\kmaxs{i}{{e \in q'}} \delta_{e} \\ = \ & \kmax{i} \left( \{\kmaxs{j}{{e \in q_{s,v}}}  \delta_{e}: j=1,..,i\} \cup \{\kmaxs{j}{{e \in p}}  \delta_{e}: j=1,..,i\}  \cup \{\kmaxs{j}{{e \in q_{w,t}}}  \delta_{e}: j=1,..,i\}  \right) 
 \\  \leqq \ & \kmax{i} \left( \{\kmaxs{j}{{e \in q_{s,v}}}  \delta_{e}: j=1,..,i\} \cup \{\kmaxs{j}{{e \in q_{v,w}}}  \delta_{e}: j=1,..,i\} \cup \{\kmaxs{j}{{e \in q_{w,t}}}  \delta_{e}: j=1,..,i\} \right) \\= \ & \kmaxs{i}{{e \in q}} \delta_{e}.
\end{align*}
It follows $\zLA(q') \leqq \zLA(q)$ and we conclude $\zLA(q') = \zLA(q)$, because $q$ is efficient with respect to $\zLA$.
\end{proof}

\begin{theo}
 When Algorithm~\ref{alg_labeling_structure} (with Algorithm~\ref{alg_labeling_create_label} as Step~\ref{alg_labeling_structure_create_label}) stops, the permanent labels at $t$ represent a complete set of efficient solutions for Problem (\ref{problem_gamma_label_onecrit}).
\end{theo}
\begin{proof}
We have to show that each permanent label at $t$ represents an efficient path from $s$ to $t$ and that for each efficient path $q$ from $s$ to $t$ a permanent label at $t$ representing $q$ or an equivalent path exists. \\The proof of the first part is analogous to the proof in \cite{ehrgott2006multicriteria} of the multi-objective label setting algorithm by Martins \cite{martins1984multicriteria}. For substituting the lexicographic order with the aggregate cost order see \cite{iori2010aggregate}. 
\\
Now, we show that the algorithm finds a complete set of efficient solutions. Assume that we have an efficient path $q$ from $s$ to $t$, such that there is no permanent label $l$ at $t$ with label costs $z=\zLA(q)$. Consider the predecessor node $v'$ of $t$ on $q$. From Lemma~\ref{lemma_subpaths_of_gamma_labeling} it follows, that there is an efficient path $p$ from $s$ to $v'$ with $\zLA(p \cup (v',t)) = \zLA(q)$. \\
If there exists a permanent label $l'$ at $v'$ with label costs $z' = \zLA(p)$, then, at the moment when it was made permanent during the algorithm, a new label $\bar{l}$ at node $t$ with label costs $\bar{z} = z' \oplus (\nome{(v',t)}, \delta_{(v',t)})$ would have been constructed. It follows \[\bar{z} = z' \oplus (\nome{(v',t)}, \delta_{(v',t)}) = \zLA(p) \oplus (\nome{(v',t)}, \delta_{(v',t)}) = \zLA(p\cup (v',t)) = \zLA(q).\] Consider the first label with cost $\zLA(q)$ that was constructed at node $t$. If this label was deleted again, its costs are dominated, which contradicts the efficiency of $q$. If it was not deleted, then it was made permanent, which contradicts our assumption that no permanent label with costs $\zLA(q)$ exists at $t$.  \\
Therefore, there is no permanent label at the predecessor node $v'$ of $t$ with costs $z'$ such that $z'\oplus(\nome{e},\delta_e) = \zLA(q)$. In the same way, we can show that there is no permanent label at the predecessor node $v''$ of $v'$ with costs $z''$ such that \[\left(z''\oplus (\nome{(v'',v')},\delta_{(v'',v')}) \right) \oplus (\nome{(v',t)},\delta_{(v',t)}) =z' \oplus (\nome{(v',t)},\delta_{(v',t)}) = \zLA(q).\] 
By induction it follows that there is no permanent label at node $s$ with cost $(0,...,0)$, which is a contradiction, because such a label is constructed in Line~\ref{alg_labeling_structure_create_first_label} and made permanent during the first execution of Line~\ref{alg_labeling_structure_choose_label}.
\\We conclude that for each efficient path $q$ from $s$ to $t$ there exists a permanent label at $t$ representing $q$ or a path that is equivalent to $q$. Furthermore, each permanent label at $t$ represents an efficient path from $s$ to $t$. Therefore, the paths represented by the permanent labels are a complete set of efficient solutions. 
\end{proof}

To find a robust optimal solution (in the multi-objective case: a complete set of robust efficient solutions) we have to filter the solutions obtained by the labeling algorithm (see Algorithm~\ref{alg_labeling+filter}).

%\marie{Ich finde es bei Algorithmen immer extrem hilfreich, wenn Input und Output drinsteht (kann auch weniger formell sein)} \andrea{ich auch!}
\begin{algorithm} 
\caption{\Lab{} for the shortest path problem with cardinality-constrained uncertainty}\label{alg_labeling+filter}
	\begin{algorithmic}[1]
	\Require an instance $I=(E, Q,\nom, \delta, \Gamma)$ of the (single-objective) cardinality-constrained uncertain shortest path problem
	\Ensure a robust optimal solution for $I$
	 \State Solve Problem~\eqref{problem_gamma_label_onecrit} with Algorithm~\ref{alg_labeling_structure}.
	 \State Compare the aggregate cost %$\sum_{i=0,...,\Gamma} z_i$ 
	 of all permanent labels in $t$ to find a minimal one. 
	 \State Obtain the path represented by this label by backtracking the predecessor labels.
	\end{algorithmic}
\end{algorithm} 
 
 \begin{cor} 
  Algorithm~\ref{alg_labeling+filter} finds an optimal solution for (\ref{problem_minmax_onecrit}) with non-negative edge costs.
 \end{cor}

\begin{remark}
 The algorithms in this section can easily be extended for the multi-objective case. We will also use the abbreviation \Lab{} to refer to the multi-objective version.
\end{remark}

\section{Experiments} \label{section_experiments}
In this paper we have presented two approaches to solve multi-objective, cardinality-constrained uncertain combinatorial optimization problems. \DSA{} solves the uncertain problem, assuming that we know how to solve the deterministic multi-objective problem. To use the \la{}  we need a method to solve a deterministic multi-objective problem with several objective functions, some of which are sums and some are bottleneck functions. We have introduced such an algorithm for the shortest path problem  (\Lab{}) and, hence, we test our approaches on the multi-objective uncertain shortest path problem.

\subsection{Hazardous material transportation}
We test our algorithms for the  multi-objective uncertain shortest path problem on a hazardous material transportation instance: When transporting hazardous materials, on one hand, the shipping company wants to minimize travel time, distance or fuel costs. On the other hand, if an accident happens, environment and population are exposed to the hazardous material, hence, another objective is to keep the risk and negative impacts of accidents to a minimum. An overview about objectives for hazardous material transportation and about approaches for estimating the risk and the impacts of an accident is given in \cite{erkut2007hazardous}.\\
For our experiments we consider the travel time and the population affected by a potential accident. We assume a nominal travel time on each road and a potential delay resulting from congestion or incidents like accidents or road construction works on some of the roads. We further assume a nominal population level, which can be increased locally by events like fairs or sport events, or due to regular shifts in population during the workday. \\
Our problem instance for hazardous material transportation is based on the instance used in \cite{robipa} to test an algorithm for bi-objective shortest path problems with only one uncertain objective. The underlying network is a sector of the Chicago region road network available at \cite{chicago} (Chicago-regional). The sector contains 1301 nodes and 4091 edges.\\ 
To obtain plausible travel times in \cite{robipa} a traffic assignment problem is solved with an iterative algorithm. It models the simultaneous movement of network users, assuming travelers follow their shortest paths. Congestion effects are taken into account by a nonlinear relationship between the flow on an edge and the travel time. Until an equilibrium solution is found, each iteration of the algorithm produces a flow and resulting travel times on the edges. To obtain the lower (upper) limit of the travel time interval for each edge we choose the smallest and largest travel times obtained during several stages of the iterative equilibrium algorithm. \\ For the population we use the distribution of the population described in \cite{robipa} as nominal values (lower interval limits). We randomly assign integer interval lengths ($\delta_{e,2}$) up to $x\%$ of the respective nominal value. By varying $x$ we obtain several test instances. We call $x$ the \emph{population uncertainty}.
\\We choose an appropriate start and end node 
with an agglomeration of population between them. Figure \ref{fig_2paths} shows two exemplary robust efficient paths for the instance with $x=10$  and $\Gamma = (5,5)$. One of the paths goes directly through the area with high population, here the time objective function has a small value, whereas the number of people exposed to the risk of health damage in case of an accident is relatively high. The other path avoids highly populated areas, which results in a longer travel time.

\begin{figure}
\centering
 \includegraphics[width=10cm]{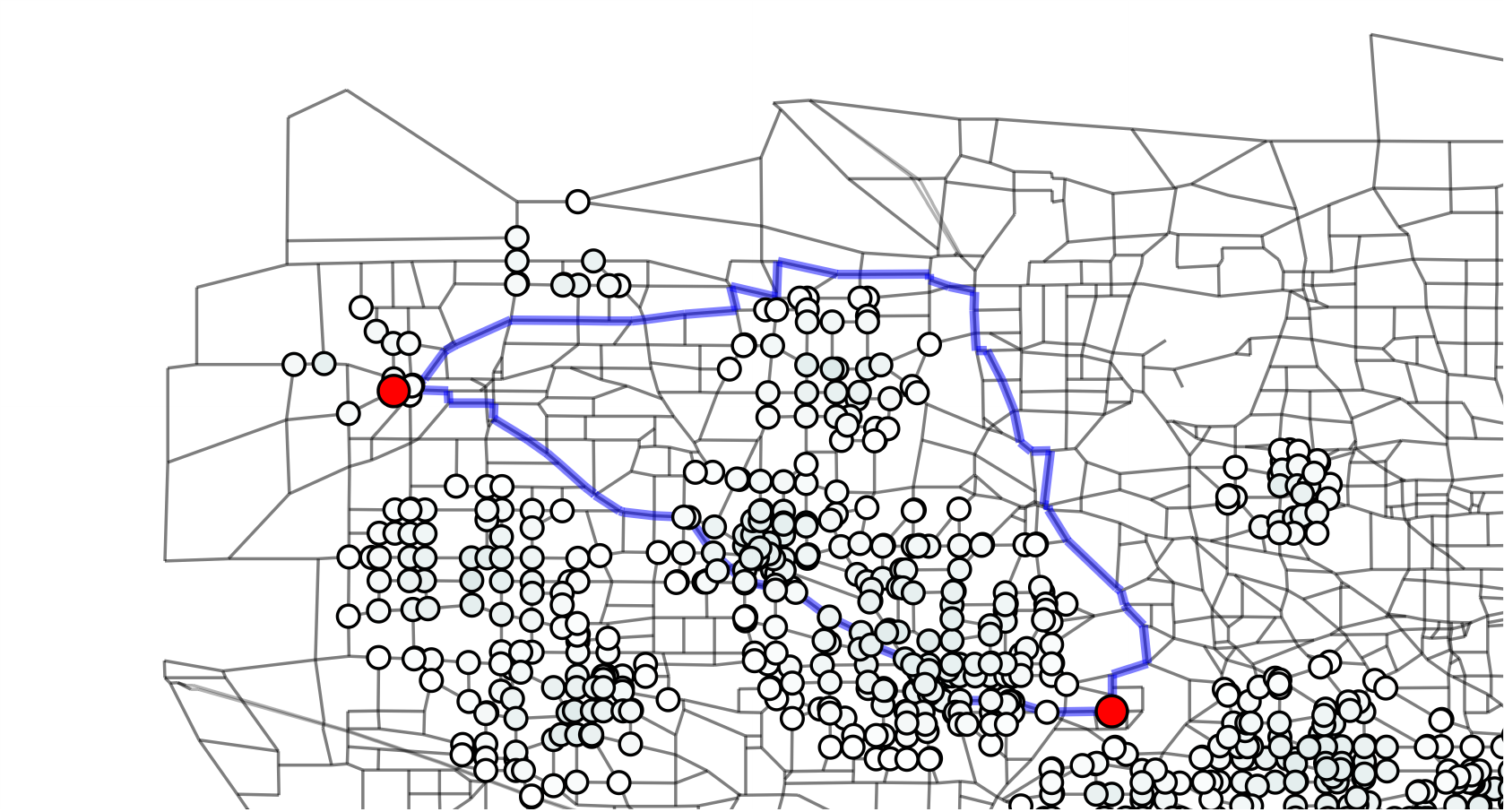}
 \caption{Section of the Chicago regional road network with distribution of population from \cite{robipa}. The red dots show start and end node chosen for our experiments and two exemplary robust efficient paths are marked in blue.}\label{fig_2paths}
\end{figure}

\subsection{Results}
The algorithms are implemented in C++, compiled under Debian 8.6 with g++ 4.9.2 compiler, and run on a Laptop with 2.10 GHz quad core processor and 7.71 GB of RAM. If not stated otherwise, we use an implementation of \DSA{} that contains all enhancements described in Section~\ref{section_DSA} and uses the special version \DSA{}-oi for instances with objective independent element order (see Section~\ref{subsection_detsubproblems_objective_independen}). For solving the subproblems we use an implementation of the Algorithm of Martins \cite{martins1984multicriteria} (with the difference that the labels are selected w.r.t. their aggregate cost instead of using the lexicographic order). There and in the implementation of \Lab{}, we additionally delete new labels at any node if they are dominated by an existing label at $t$. \\In the figures, one data point represents one measurement, except for Section~\ref{subsection_exp_corr}, where we took the average running time of 40 runs. 
\\
To compare the performance of our solution approaches, we solve the bi-objective hazardous material transportation instance described above for different values of population uncertainty $x$ and $\Gamma$ (to keep the number of parameters low, we always choose the same value for  $\Gamma_1$ and $\Gamma_2$ and we will refer to this value as $\Gamma_i$ in the following). In addition, we compare the performance of the algorithms on an instance with three objectives and on an instance with two correlated objective functions. We further evaluate the improvement gained by our enhancement of \DSA{} (solution checking). Finally, we generate an instance with objective independent element order and compare the running time of \DSA{}-oi for such instances to the general \DSA{}.
\subsubsection{Comparison of the two solution approaches for the hazardous material transportation instance}\label{subsection_exp1}
Figure~\ref{fig_compare_DSA_lab} shows the running time of \DSA{} and \Lab{} for several values of $\Gamma_i$ and $x$. 
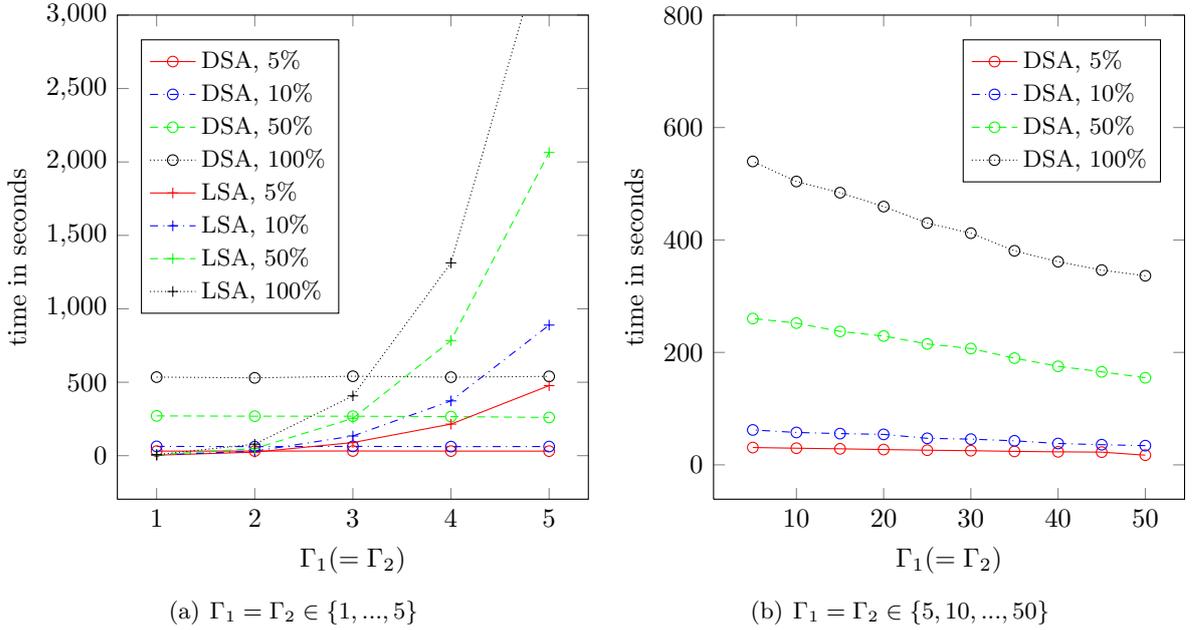
\begin{figure}
%\centering
\subfigure[{$\Gamma_1 = \Gamma_2  \in \{1,...,5\}$}]{
\begin{tikzpicture}

\begin{axis}
[width=0.5\textwidth,height=8cm, xlabel = {$\Gamma_1 (= \Gamma_2)$}, ylabel = time in seconds, legend entries = {{\DSA{}, 5\%}, {\DSA{}, 10\%}, {\DSA{}, 50\%}, {\DSA{}, 100\%}, {\Lab{}, 5\%}, {\Lab{}, 10\%}, {\Lab{}, 50\%}, {\Lab{}, 100\%}},
 legend style = { at = {(0.05,0.95)}, anchor = north west, font=\footnotesize},legend cell align = left, ymax = 3000]
\pgfplotstableread{plotdata/dsa_labeling.csv}
\datatable
\addplot[color=red, mark=o] table[y = DSA_5] from \datatable ;
\addplot[mark=o, color=blue,style =dashdotted , mark options=solid] table[y = DSA_10] from \datatable ;
\addplot[color=green, mark=o, style=densely dashed, mark options=solid] table[y = DSA_50] from \datatable ;
\addplot[mark=o, color=black, densely dotted, mark options=solid] table[y = DSA_100] from \datatable ;
\addplot[color=red, mark=+] table[y = Labeling_5] from \datatable ;
\addplot[mark=+, color=blue, style=dashdotted, mark options=solid] table[y = Labeling_10] from \datatable ;
\addplot[color=green, mark=+, style=densely dashed, mark options=solid] table[y = Labeling_50] from \datatable ;
\addplot[mark=+, color=black, densely dotted, mark options=solid] table[y = Labeling_100] from \datatable ;
\end{axis}
\end{tikzpicture}
}
\subfigure[$\Gamma_1 = \Gamma_2 \in \{5,10,...,50\}$]{
\begin{tikzpicture}
\begin{axis}
[width=0.5\textwidth,height=8cm,xlabel = {$\Gamma_1 (= \Gamma_2)$}, ylabel = time in seconds, legend entries = {{\DSA{}, 5\%}, {\DSA{}, 10\%}, {\DSA{}, 50\%}, {\DSA{}, 100\%}},
 legend style = { at = {(0.95,0.95), anchor = north east},font=\footnotesize}, legend cell align = left, ymax = 800]
\pgfplotstableread{plotdata/dsa_subpr_time_50.csv}
\datatable
\addplot[color=red, mark=o] table[y = mc5] from \datatable ;
\addplot[color=blue, mark=o, style=dashdotted, mark options=solid] table[y = mc10] from \datatable ;
\addplot[mark=o, color=green, style=densely dashed, mark options=solid] table[y = mc50] from \datatable ;
\addplot[mark=o, color=black, style= densely dotted, mark options=solid] table[y = mc100] from \datatable ;
\end{axis}
\end{tikzpicture}
}
\caption{Running time of \DSA{} and \Lab{} for several values of $\Gamma_i$ and population uncertainty $x$ on two different scales.}\label{fig_compare_DSA_lab}
\end{figure}
The running time of \Lab{} increases with $\Gamma_i$, whereas the running time of \DSA{} decreases (see also Figure~\ref{fig_DSA_check}).  The reason is that for increasing $\Gamma_i$, the number of objectives in the deterministic multi-objective problem solved during \Lab{}  increases as well, whereas the maximal number of subproblems solved during \DSA{} decreases. 
For small values of $\Gamma_i$ \Lab{} solves the given instances faster, for higher values \DSA{} has a better performance. %For $\Gamma = (6,6)$, the running time of \Lab{} exceeds ....
\\If we choose a higher value for $x$, which results in a greater maximal and mean deviation from the nominal value and a higher number of different values of $\delta_{e,2}$, the running time of both algorithms increases. In the case of \DSA, the increase of the running time can be explained by the higher number of different values of $\delta_{e,2}$, which leads to a higher number of subproblems.
\subsubsection{Three objectives}
Since we are also interested in the performance of the algorithms for problems with more than two objectives, we generate an artificial third objective using, again, the nominal population.  We generate random interval lengths in the same range as the other population objective, i.e., the value of population uncertainty in general is the same for both population objectives, but the specific interval lengths of each edge may differ. Figure~\ref{fig_compare_DSA_lab_3obj} shows the running times on this instance in comparison to the instance with two objectives described above.

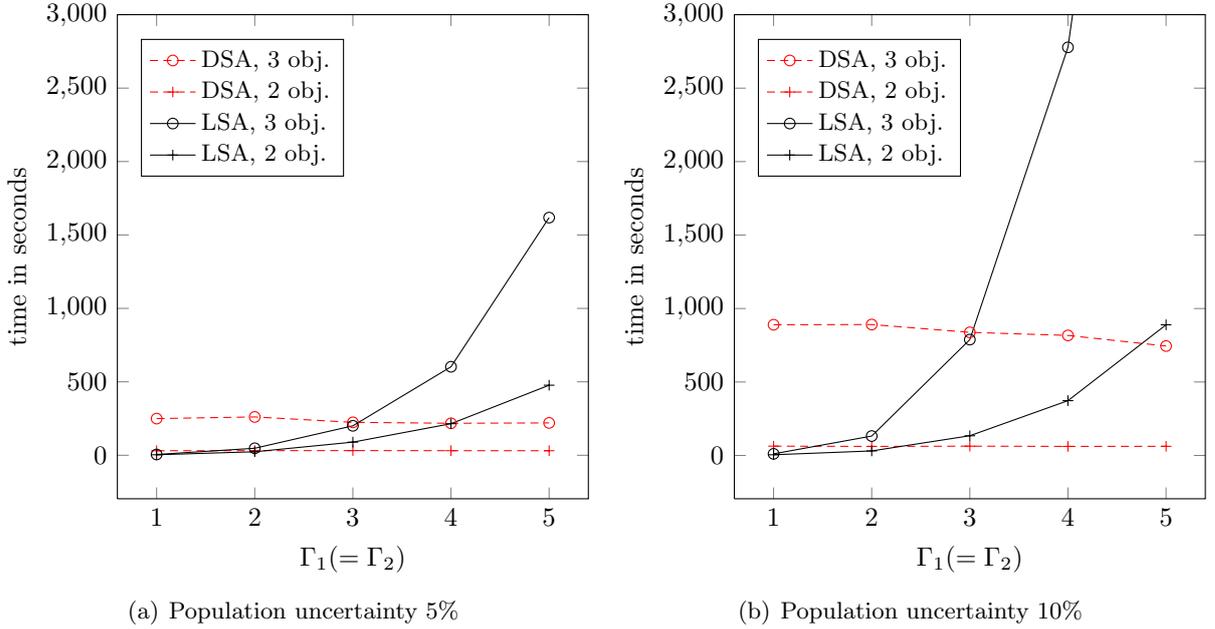
\begin{figure}
%\centering
\subfigure[Population uncertainty 5\%]{
\begin{tikzpicture}
\begin{axis}
[width=0.5\textwidth,height=8cm,xlabel = {$\Gamma_1 (= \Gamma_2)$}, ylabel = time in seconds, legend entries = {{\DSA{}, 3 obj.}, {\DSA{}, 2 obj.}, {\Lab{}, 3 obj.}, {\Lab{}, 2 obj.}},
 legend style = { at = {(0.05,0.95)}, anchor = north west}, legend cell align = left, ymax = 3000]
\pgfplotstableread{plotdata/dsa_labeling_3obj.csv}
\datatable
\addplot[color=red, mark=o,style=densely dashed, mark options=solid] table[y = DSA_5_3obj] from \datatable ;
\addplot[mark=+, color=red,style=densely dashed, mark options=solid] table[y = DSA_5] from \datatable ;
\addplot[color=black, mark=o] table[y = Labeling_5_3obj] from \datatable ;
\addplot[mark=+, color=black] table[y = Labeling_5] from \datatable ;
\end{axis}
\end{tikzpicture}
}
\subfigure[Population uncertainty 10\%]{
\begin{tikzpicture}
\begin{axis}
[width=0.5\textwidth,height=8cm, xlabel = {$\Gamma_1 (= \Gamma_2)$}, ylabel = time in seconds, legend entries = {{\DSA{}, 3 obj.}, {\DSA{}, 2 obj.}, {\Lab{}, 3 obj.}, {\Lab{}, 2 obj.}},
 legend style = { at = {(0.05,0.95)}, anchor = north west}, legend cell align = left, ymax = 3000]
\pgfplotstableread{plotdata/dsa_labeling_3obj.csv}
\datatable
\addplot[color=red, mark=o,style=densely dashed, mark options=solid] table[y = DSA_10_3obj] from \datatable ;
\addplot[mark=+, color=red,style=densely dashed, mark options=solid] table[y = DSA_10] from \datatable ;
\addplot[color=black, mark=o] table[y = Labeling_10_3obj] from \datatable ;
\addplot[mark=+, color=black] table[y = Labeling_10] from \datatable ;
\end{axis}
\end{tikzpicture}
}
\caption{Running time of \DSA{} and \Lab{} for an instance with three objectives and an instance with two objectives.}\label{fig_compare_DSA_lab_3obj}
\end{figure}
The running time of both algorithms increases by including the additional objective. The relative difference between the running time of the instance with two objectives and the instance with three objectives increases with $\Gamma_i$ for \Lab{}, whereas it decreases for \DSA.
\subsubsection{Correlated objective functions}\label{subsection_exp_corr}
We additionally generate an instance with two strongly correlated objective functions: We use the travel time as one objective and generate a second travel time objective by multiplying the nominal times and the interval lengths each by a random factor between $0.9$~and~$1.1$. 

\begin{figure}
\centering
\begin{tikzpicture}
\begin{axis}
[width=0.5\textwidth,xlabel = {$\Gamma_1 (= \Gamma_2)$}, ylabel = time in seconds, legend entries = {\DSA{}, \Lab{}}, legend style = { at = {(0.95,0.05)}, anchor =south east}, legend cell align = left, ymax = 4]
\pgfplotstableread{plotdata/dsa_lab_corr.csv}
\datatable
\addplot[color=red, mark=o] table[y = DSA] from \datatable ;
\addplot[mark=+, color=black] table[y = Lab] from \datatable ;
\end{axis}
\end{tikzpicture}
\caption{Running time of \DSA{} and \Lab{} for an instance with two strongly correlated objective functions.
}\label{fig_corr}
\end{figure}
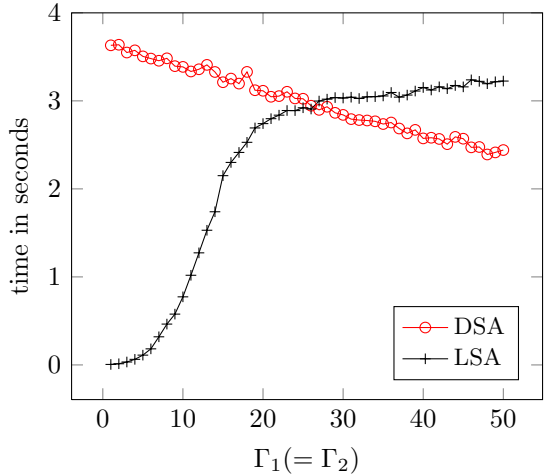
Both algorithms benefit a lot from the correlation, all running times are now less than four seconds, as shown in Figure~\ref{fig_corr}. In comparison, \Lab{} benefits more from correlated objective function values: The values of $\Gamma_i$, for which it is still faster than \DSA, are much higher on this instance than on the original hazardous material transportation instance considered in Section~\ref{subsection_exp1}. For small values of $\Gamma_i$ it is much faster than \DSA{}.
\subsubsection{Evaluation of the improvement obtained by solution checking }\label{subsection_exp_solcheck} %(Lemma~\ref{lemma_alg_Bertsim_faster_2_general_a}) }
To evaluate the obtained improvement by using solution checking in \DSA{}, we use Algorithm~\ref{alg_Bert_Sim_alternative_step3_mo2} as Step~\ref{alg_step_solve} of Algorithm~\ref{alg_Bert_Sim_general}.  We compare the running time of the version containing solution checking to the running time of the version without this enhancement, and, additionally, count the solved subproblems (Figure~\ref{fig_DSA_check}). Where fewer subproblems have been solved because of the enhancement, the running times differ significantly, for all other instances they are nearly equal. Hence, the check itself does not slow down the algorithm significantly in comparison to the acceleration that we obtain when subproblems can be skipped. We conclude that it is worth using the enhancement, but as $\Gamma_i$ increases solution checking becomes less effective.\\
Note that, since Lemma~\ref{lemma_alg_Bertsim_faster_2_general_a} allows to exclude even more subproblems than excluded in Algorithm~\ref{alg_Bert_Sim_alternative_step3_mo2}, further speed-ups may be achieved by implementing a more sophisticated solution checking. However, already when using Algorithm~\ref{alg_Bert_Sim_alternative_step3_mo2}, the benefit of solution checking is clearly visible.  

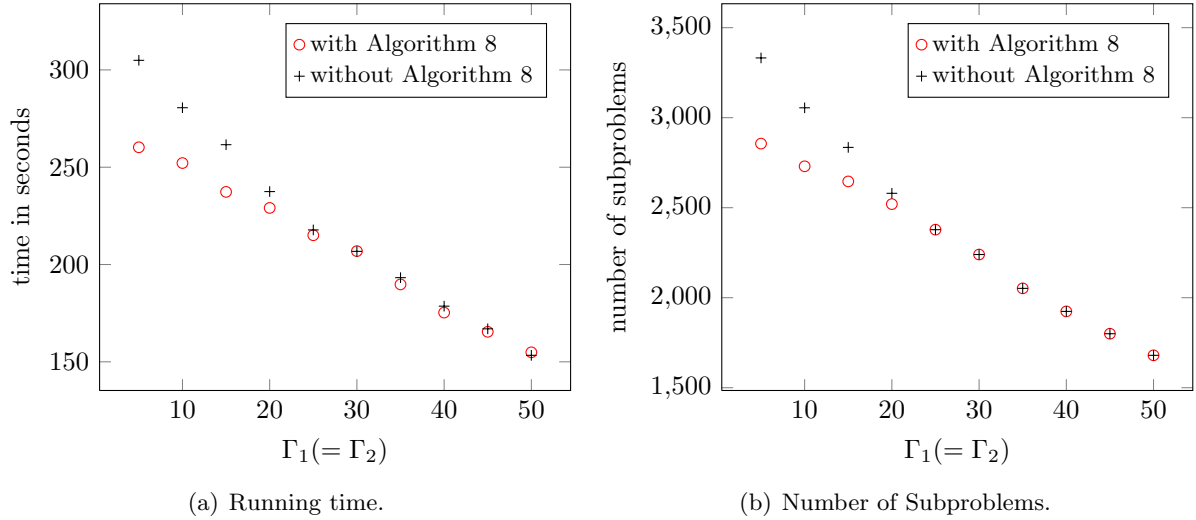
\begin{figure}
%\centering
\subfigure[Running time.]{
\begin{tikzpicture}
\begin{axis}
[width=0.5\textwidth,xlabel = {$\Gamma_1 (= \Gamma_2)$}, ylabel = time in seconds, legend entries = {with Algorithm~\ref{alg_Bert_Sim_alternative_step3_mo2}, without Algorithm~\ref{alg_Bert_Sim_alternative_step3_mo2}},
only marks, legend style = { at = {(0.95,0.95)}},legend cell align = left, ymax = 334]
\pgfplotstableread{plotdata/dsa_subpr_time_50.csv}
\datatable
\addplot[color=red, mark=o] table[y = mc50] from \datatable ;
\addplot[mark=+, color=black] table[y = oc50] from \datatable ;
\end{axis}
\end{tikzpicture}
}
\subfigure[Number of Subproblems.]{
\begin{tikzpicture}
\begin{axis}
[width=0.5\textwidth,xlabel = {$\Gamma_1 (= \Gamma_2)$}, ylabel = number of subproblems, legend entries = {with Algorithm~\ref{alg_Bert_Sim_alternative_step3_mo2}, without Algorithm~\ref{alg_Bert_Sim_alternative_step3_mo2}},
only marks, legend style = { at = {(0.95,0.95)}},legend cell align = left, ymax = 3633]
\pgfplotstableread{plotdata/dsa_subpr_time_50.csv}
\datatable
\addplot[color=red, mark=o] table[y = sp50] from \datatable ;
\addplot[mark=+, color=black] table[y = sp50all] from \datatable ;
\end{axis}
\end{tikzpicture}
}
\caption{Running time and number of solved subproblems of \DSA{} with and without solution checking (Population uncertainty 50\%).}\label{fig_DSA_check}
\end{figure}

\subsubsection{Evaluation of \DSA{} for instances with objective independent element order}
To compare the performance of the objective independent \DSA{}  (\DSA{}-oi) from Section \ref{subsection_detsubproblems_objective_independen} to the general algorithm, we generate an instance with objective independent element order: Instead of generating interval lengths for the population objective we use the interval lengths of the travel time objective. Figure~\ref{fig_DSA_equaltime} shows that \DSA{}-oi has a much better performance than the general algorithm. The test, whether the instance is objective independent, only takes a small fraction of the running time (for our instances $1.4 \cdot 10^{-5}$ seconds). Therefore, it is reasonable to check each instance for objective independent element order before solving it with \DSA{}.
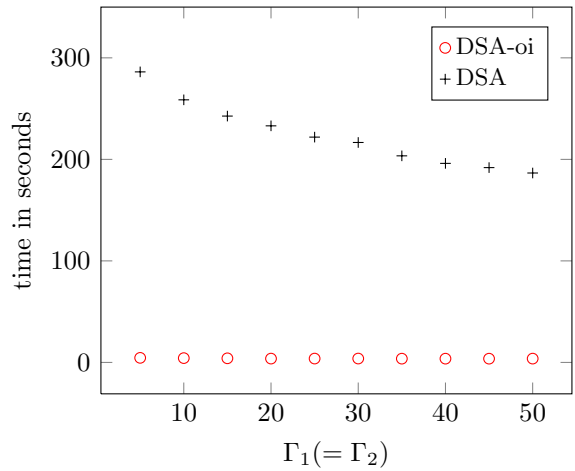
\begin{figure}
\centering
\begin{tikzpicture}
\begin{axis}
[width=0.5\textwidth,xlabel = {$\Gamma_1 (= \Gamma_2)$}, ylabel = time in seconds, legend entries = {\DSA{}-oi, \DSA{}},
only marks, legend style = { at = {(0.95,0.95)}},legend cell align = left, ymax = 350]
\pgfplotstableread{plotdata/dsa_equaltime.csv}
\datatable
\addplot[color=red, mark=o] table[y = me] from \datatable ;
\addplot[mark=+, color=black] table[y = oe] from \datatable ;
\end{axis}
\end{tikzpicture}
% \centering

\caption{Comparison of \DSA{} and \DSA{}-oi for instances with objective independent element order.}\label{fig_DSA_equaltime}
\end{figure}

\section{Conclusion}

 In this paper we have developed two approaches to find minmax robust solutions for multi-objective combinatorial optimization problems with cardinality-constrained uncertainty. We have extended an algorithm from \cite{bertsimas2003robust} (\DSA{}) to multi-objective optimization, have suggested an enhancement and developed a special version for instances with objective independent element order. We have also introduced a second approach and used it to develop a label setting algorithm (\Lab{}) for the multi-objective uncertain shortest path problem.\\
 We have tested our algorithms on several instances of the multi-objective uncertain shortest path problem arising from hazardous material transportation. On most of the tested instances \DSA{} has a better performance, but \Lab{} is faster for small values of $\Gamma_i$. If the two objective functions are strongly correlated, which appears often in shortest path problems, where, e.g., the distance, travel time and fuel consumption are correlated, \Lab{} is competitive even for higher values of $\Gamma_i$.\\
 When implementing \DSA{} we recommend to use the proposed enhancements and to check whether the special version for instances with (partly) objective independent element order can be used. The checks do not take long in comparison to the total running time, and if their result is positive, the algorithm can be accelerated significantly.\\
 For further investigations other variants of multi-objective cardinality-constrained uncertainty are of interest. A second way to extend the single-objective concept is to require the edges whose costs differ from their minimal values to be the same for all objectives. In this case the uncertainties in the objectives are no longer independent of each other and using point-based or set-based minmax robust efficiency leads to different solution sets. An interesting variation of cardinality-constrained uncertainty is not to consider a bound on the cardinality, but on the sum of the deviation from their minimal values. \\
 Further research on robust multi-objective optimization includes other types of uncertainty, e.g., discrete scenario sets or polyhedral or ellipsoidal uncertainty. Also the case of decision uncertainty, in which the solution found can not be realized exactly, is of interest, see \cite{RegRob} for first results.\\
 The algorithms for the multi-objective cardinality-constrained uncertain shortest path problem presented in this paper can easily be extended to the \emph{multi-objective single-source shortest path problem}, where a complete set of efficient paths from a start node $s$ to all other nodes has to be found. Since, in the deterministic case, there exist algorithms (e.g. the algorithm of Martins \cite{martins1984multicriteria}) for which it can be shown that the running time is polynomial in the output size, it would be interesting to investigate whether this is the case for the uncertain problem, too.

\section*{Acknowledgments}
Lisa Thom was supported by DFG RTG 1703 ``Resource Efficiency in Interorganizational Networks''.

\bibliography{romupa}
\bibliographystyle{alpha}

\end{document}